\documentclass{article}  
\usepackage{floatrow}
\newfloatcommand{capbtabbox}{table}[][\FBwidth]
\usepackage[pdftex,
bookmarksnumbered,
bookmarksopen,
pagebackref,
colorlinks,
citecolor=blue,
linkcolor=blue,]{hyperref}
\usepackage{amsmath}
\usepackage{amssymb}
\usepackage{mathrsfs}
\usepackage{amsfonts}
\usepackage{amsthm}
\usepackage{algorithm}
\usepackage{algorithmic}
\usepackage{booktabs}
\usepackage{listings}
\usepackage{boxedminipage}
\usepackage{stmaryrd}
\usepackage{cite}
\usepackage{pgfplots}
\usepackage{tikz}
\usepackage{tkz-tab}
\usepackage{relsize}
 \usepackage{lscape}
\usepackage[top=1.3in, bottom=1.5in, left=1.5in, right=1.5in]{geometry}
\usepackage{color,xcolor}

\newtheorem{theorem}{Theorem}[section]
\newtheorem{definition}{Definition}[section]
\newtheorem{proposition}{Proposition}[section]

\newtheorem{lemma}{Lemma}[section]

\newtheorem{remark}{Remark}[section]

\newcommand{\normF}[1]{\ensuremath{\| #1 \|_{F}  }}    

\newcommand{\T}{\ensuremath{ \mathbb R^{n_1\times \cdots\times n_d}   }}

\newcommand{\bigxiaokuohao}[1]{\ensuremath{ \left(  #1 \right) }}      
\newcommand{\bigjueduizhi}[1]{\ensuremath{ \left|  #1 \right| }}   
         
\newcommand{\bigzhongkuohao}[1]{\ensuremath{ \left[   #1 \right] }}      

\newcommand{\bigfnorm}[1]{\ensuremath{ \left\|   #1 \right\|_F }}    
\newcommand{\bignorm}[1]{\ensuremath{ \left\|   #1 \right\|  }}        
\newcommand{\bigllbracket}[1]{\ensuremath{ \left\llbracket   #1 \right\rrbracket }}      

\newcommand{\innerprod}[2]{\ensuremath{ \left\langle   #1 , #2\right\rangle }}      

\newcommand{\bigotimesu}{\ensuremath{     \bigotimes^d_{j=1}\nolimits \mathbf u_{j,i  } }}

\newcommand{\tensorsigmaU}{\ensuremath{     \llbracket \boldsymbol{ \sigma};U_j \rrbracket                             }}   
 
\newcommand{\stmanifold}[2]{\ensuremath{     {\rm st}(#1,#2)  }}

\newcommand{\deltaUjP}[2]{\ensuremath{ \Delta_{U_j,\mathcal T}^{#1,#2}   }}      
\newcommand{\deltaUj}[2]{\ensuremath{ \Delta_{U_j}^{#1,#2}   }}      
      
\newcommand{\deltasigma}[2]{\ensuremath{ \Delta_{\boldsymbol{ \sigma}}^{#1,#2}   }}      
\newcommand{\deltaP}[2]{\ensuremath{ \Delta_{\mathcal T}^{#1,#2}   }}      
\newcommand{\deltaY}[2]{\ensuremath{ \Delta_{\mathcal Y}^{#1,#2}   }}      
\newcommand{\deltaW}[2]{\ensuremath{ \Delta_{\mathcal W}^{#1,#2}   }}      
\newcommand{\tildeLtaualpha}[2]{\ensuremath{ \tilde L_{\tau,\alpha}^{#1,#2}    }}

	\definecolor{darkgray}{rgb}{0.66, 0.66, 0.66}

\newenvironment{mytabular}{\bgroup\tiny\tabular}{\endtabular\egroup}
\newenvironment{mytabular1}{\bgroup\footnotesize\tabular}{\endtabular\egroup}

\title{Half-Quadratic Alternating Direction Method of Multipliers for Robust Orthogonal   Tensor Approximation}

\author{Yuning Yang\thanks{College of Mathematics and Information Science, Guangxi University, Nanning, 530004, China  (yyang@gxu.edu.cn).} \and Yunlong Feng\thanks{Department of Mathematics and Statistics, State University of New York at Albany, Albany, New York 12222, USA (ylfeng@albany.edu).}
}

\begin{document} 
\maketitle

\begin{abstract}
  Higher-order tensor canonical polyadic decomposition (CPD) with one or more of the latent factor matrices being columnwisely orthonormal has been well studied in recent years. However, most existing models penalize the noises, if occurring, by employing the least squares loss, which may be sensitive to non-Gaussian noise or outliers, leading to   bias estimates of the latent factors. In this paper, based on the maximum a posterior estimation, we derive a robust orthogonal tensor CPD model with Cauchy loss, which
    is resistant to   heavy-tailed noise or outliers. By exploring the  half-quadratic property of the model, a new method, which is termed as half-quadratic alternating direction method of multipliers (HQ-ADMM), is proposed to solve the model. Each subproblem involved in HQ-ADMM admits a closed-form solution. Thanks to some nice properties of the Cauchy loss, we show that the whole sequence generated by the algorithm globally converges to a stationary point of the problem under consideration. Numerical experiments on synthetic and real data demonstrate the efficiency and robustness of the proposed model and algorithm. 

\hspace{2mm}\vspace{3mm}

\end{abstract}

\section{Introduction} \label{sec:intro}
A tensor is a multidimensional array. Owing to its ability to represent data with intrinsically many dimensions, tensors draw much attention from the communities of signal processing, image processing, machine learning, etc; see the surveys \cite{kolda2010tensor,cichocki2015tensor,sidiropoulos2017tensor}. To understand the relationship behind the data tensor, decomposition  tools are needed. In general, tensor decomposition aims at factorizing the data tensor into a set of lower-dimensional latent factors, where the factors can be vectors, matrices or even tensors. Among the decomposition models, tensor canonical polyadic decomposition (CPD), which factorizes a tensor into a sum of component rank-$1$ tensors,  is one of the most important   models. Tensor CPD finds applications in blind multiuser CDMA, blind source separation, and so on \cite{sidiropoulos2017tensor}. Different from matrix decompositions, tensor CPD is  unique  under quite mild conditions \cite{kolda2010tensor}. 
 
 In some applications, one or more latent factors of the CPD are required to have orthonormal columns. For example, in linear image coding \cite{shashua2001linear}, one is given a set of data matrices of the same size; to explore their commonalities, one   projects the matrices onto a latent lower-dimensional subspace in which the subspace can be represented by the Khatri-Rao product \cite{kolda2010tensor} of two columnwisely orthonormal matrices. Such a problem has been formulated as a third-order tensor CPD with two factor matrices having orthonormal columns. On the other hand, simultaneous foreground-background extraction and compression can also be formulated as a model of the same kind; this will be illustrated in Sect. \ref{sec:numer}.
 Other applications of CPD with orthonormal factors can be found in \cite{de2011short,de2010algebraic,sidiropoulos2000blind,de2013joint,sorensen2010parafac}.
 
 In reality, due to the NP-hardness of determining the tensor rank \cite{hillar2013most}, and due to the presence of noise, tensor CPD model with orthonormal factors is rarely exact, and it is necessary to resort to an approximation scheme. To numerically solve the problem, one usually formulates it as an optimization problem that minimizes the Euclidean distance between the data tensor and the latent tensor over orthonormal constraints, and then applies    an alternating optimization  type  method to solve it  based on polar decomposition   \cite{chen2009tensor,sorensen2012canonical,wang2015orthogonal,pan2018symmetric,guan2019numerical,yang2019epsilon,hu2019linear,li2019polar}. Other types of methods can be found in \cite{ishteva2013jacobi,li2018globally,savas2010quasi,de2000a}; just to name a few.
 
 Although the  optimization model mentioned above is effective in some circumstances, note that the Euclidean distance,   built upon the least squares loss that is not robust \cite{huber2004robust}. As a result, when the data tensor is contaminated by heavy-tailed noise or outliers, such least squares based models     often lead  to   bias estimates of the true latent factors, as having been observed   in practice.    This drawback of the least squares based   models motivates  us to develop a new model that is robust to heavy-tailed noise or outliers. 
 
 In this work, from the maximum a posterior estimation, we derive a robust tensor CPD model where   one or more latent  factors have orthonormal columns. Such a model is based on the Cauchy loss, whose robustness comes from the redescending property of the loss function, as pointed out in robust statistics \cite{huber2004robust}. We then explore the half-quadratic property of the model, based on which, the half-quadratic alternating direction method of multipliers (HQ-ADMM)  is proposed to solve the model. An advantage of HQ-ADMM is that every subproblem involved in the algorithm admits a closed-form solution. Under a very mild assumption on the parameter, HQ-ADMM is proved to globally converge to a stationary point of the problem under consideration, owing to some nice properties of the Cauchy loss. In fact, the spirit of HQ-ADMM can be extended to solving other Cauchy loss based machine learning and scientific computing problems (besides tensor problems), which will be remarked later in Sect. \ref{sec:alg}. Finally, we show via numerical experiments that the proposed model is resistant to heavy-tailed noise such as Cauchy noise, outliers, and also performs well with Gaussian noise; the proposed HQ-ADMM is observed to be efficient.
 
 The rest of the paper is organized as follows. The robust tensor approximation model is formulated  in Sect. \ref{sec:problem}, with some quantitative properties given.  The HQ-ADMM is developed in Sect. \ref{sec:alg}; the convergence analysis of HQ-ADMM is provided in Sect. \ref{sec:conv}. Numerical results are illustrated in Sect. \ref{sec:numer}. We end this paper in Sect. \ref{sec:conclusion} with conclusions.

\section{Problem Formulation and the Optimization Model} \label{sec:problem}
 \paragraph{Notations} Vectors are written as boldface lowercase letters $(\mathbf x,\mathbf y,\ldots)$, matrices
 are denoted as italic capitals $(A,B,\ldots)$, and tensors are
 written as calligraphic capitals $(\mathcal{A}, \mathcal{B},
 \cdots)$. $\mathbb R$   denotes  the real   field.  $\mathbb R^{m\times n}$ denotes real matrices of dimension $m\times n$ and $\T$ denotes tensor space of size $n_1\times\cdots \times n_d$. The Frobenius norm, $\bigfnorm{\cdot}$, of a matrix or a tensor, is defined to be the square root of the sum of squares of all the entries. The inner product $\innerprod{\cdot}{\cdot}$ between a pair of matrices or tensors of the same size is given by the sum of entrywise product.
$\otimes$ denotes the outer product of two vectors. Other notations will be introduced whenever necessary.

Let $\mathcal A=\bigxiaokuohao{\mathcal A_{i_1\cdots i_d}}\in \T$ be a $d$-th order observed data tensor. We consider the inexact CPD of $\mathcal A$, i.e., approximating $\mathcal A$ by   a sum of rank-1 tensors:
\begin{equation}\label{prob:cpd}			      \setlength\abovedisplayskip{2pt}
\setlength\abovedisplayshortskip{2pt}
\setlength\belowdisplayskip{2pt}
\setlength\belowdisplayshortskip{2pt}
\mathcal A =  \sum^R_{i=1}\nolimits\sigma_i \bigotimes^d_{j=1}\nolimits \mathbf u_{j,i  } + \mathcal N \in \T;
\end{equation}
here $\mathbf u_{j,i}\in\mathbb R^{n_j},1\leq j\leq d$, $\bigotimes_{j=1}^d\mathbf u_{j,i}$ denotes the rank-1 tensor  given by the outer product of $\mathbf u_{j,i}$'s, $\sigma_i$'s are real scalars,  $R>0$ is a given integer, where usually  $R$ is such that   $R\leq \min\{n_1,\ldots,n_d \}$ for a possibly low-rank approximation, while $\mathcal N$ denotes the noisy tensor. 

Denote   $U_j
 := \bigzhongkuohao{ \mathbf u_{j,1},\ldots, \mathbf u_{j,R}} \in\mathbb R^{n_j\times R}$ and $\boldsymbol{ \sigma}:=[\sigma_1,\ldots,\sigma_R]\in\mathbb R^R$. Then $U_j$'s are called the latent factor matrices of $\mathcal A$. Throughout this work, we follow \cite{kolda2010tensor} to write the sum of rank-1 terms as
 \[
 \llbracket \boldsymbol{ \sigma}; U_1,\ldots,U_d\rrbracket := \sum^R_{i=1}\nolimits\sigma_i \bigotimes^d_{j=1}\nolimits\mathbf u_{j,i};
 \]
 moreover, we write $\llbracket \boldsymbol{ \sigma}; U_j\rrbracket :=  \llbracket \boldsymbol{ \sigma}; U_1,\ldots,U_d\rrbracket $ for short. 
In the sequel,   we base our work on the following setup:
 	\begin{itemize}
 		\item One or more $U_j$'s are columnwiely orthonormal. Without loss of generality, we assume that   the last $t$ $(1\leq t\leq d)$ matrices are columnwisely orthonormal, i.e., $$U_j^\top U_j = I,~d-t+1\leq j\leq d,$$
 		where $I $ is an identity matrix of the proper size;
 		\item The columns of the first $d-t$ matrices are normalized, i.e.,
 		\[
 		\bignorm{\mathbf u_{j,i}} = 1,~1\leq j\leq d-t,1\leq i\leq R;
 		\]
 		\item Entries of the noisy tensor $\mathcal N$ are i.i.d..
 	\end{itemize}
We immediately have the following proposition.
\begin{proposition}\label{prop:orth}
	There holds $\bigfnorm{\bigotimes_{j=1}^d\mathbf u_{j,i} }=1$, $1\leq i\leq R$, and $ \innerprod{\bigotimes_{j=1}^d\mathbf u_{j,i_1}   }{\bigotimes_{j=1}^d\mathbf u_{j,i_2}  }=0,~i_1\neq i_2.  $
\end{proposition} 
Note that the   constraints on $\mathbf u_{j,i} $ and $U_j$ are all   Stiefel manifolds ${\rm st}(m,n):= \{P\in\mathbb R^{m\times n}\mid P^\top P = I  \}$. Therefore, in the following, we write the constraints on $\mathbf u_{j,i}$ and $U_j$ as
\begin{equation*}
\begin{split}
&\mathbf u_{j,i} \in \stmanifold{n_j}{1},~1\leq j\leq d-t,1\leq i\leq R,\\
& U_j\in\stmanifold{n_j}{R},~d-t+1\leq j\leq R.
\end{split}
\end{equation*}

In the presence of the noisy term $\mathcal N$, it is   natural to deal with \eqref{prob:cpd} via solving the following optimization problem \cite{chen2009tensor,sorensen2012canonical,wang2015orthogonal,guan2019numerical,yang2019epsilon}:
\begin{equation}			      \setlength\abovedisplayskip{2pt}
\setlength\abovedisplayshortskip{2pt}
\setlength\belowdisplayskip{2pt}
\setlength\belowdisplayshortskip{2pt}
\label{prob:obj_ls}
\min_{\boldsymbol{ \sigma}, \mathbf u_{j,i}\in\stmanifold{n_j}{1},U_j\in\stmanifold{n_j}{R}} \bigfnorm{ \mathcal A -  \bigllbracket{ \boldsymbol{\sigma}; U_j  }      }^2 =  \sum^{n_1,\ldots,n_d}_{i_1=1,\ldots,i_d=1}\bigxiaokuohao{  \mathcal A_{i_1\cdots i_d} -  \bigllbracket{ \boldsymbol{\sigma}; U_j }_{i_1\cdots i_d}     }^2.
\end{equation}
 From a statistical estimation viewpoint, the above model is built upon the least squares loss $\ell_2(t) := t^2/2$, i.e., it employs the $\ell_2(\cdot)$ loss to deal with noise. However, it is commonly known that the estimators induced by the least squares loss are sensitive to heavy-tailed noise or outliers; in other words, by using the model \eqref{prob:obj_ls}, one   assumes that every entry of $\mathcal N$ obeys the standard Gaussian distribution by default.
 
\paragraph{Derivation of our model} In real-world applications, data may be contaminated by heavy-tailed noise, and even outliers/impulsive noise. A typical non-Gaussian and heavy-tailed noise is the Cauchy noise, whose probability density function is given by
 \begin{equation*}\label{eq:cauchy_noise}
 P_{{\rm Cauchy} }(t) \propto \frac{1}{1+ (t-c)^2/\delta^2},
 \end{equation*}
 where $\delta>0$ is the scale parameter and $c$ is the location paramter. By assuming the symmetry of the noise, we let $c=0$ in the above function.  
 
We derive our model from the maximum a posterior (MAP) estimation by assuming that $\mathcal N$ obeys the Cauchy distribution whose density function is given above. To this end, denote respectively the indicator function $\boldsymbol{1}_C(\cdot)$ and the characteristic function $\iota_C(\cdot)$ of a closed set $C$ as follows
\begin{equation*}
\begin{split}
& \boldsymbol{1}_C(\mathbf x) = 1,~{\rm if}~\mathbf x\in C;~ \boldsymbol{1}_C(\mathbf x) = 0,~{\rm if}~\mathbf x\not\in C,\\
& \iota_C(\mathbf x) = 0,~{\rm if}~\mathbf x\in C;~ \iota_C(\mathbf x) = +\infty,~{\rm if}~\mathbf x\not\in C.
\end{split}
\end{equation*}
From the constraints on $\mathbf u_{j,i}$ and $U_j$,   it is natural to impose a uniform prior belief distributional assumption on   $\{\mathbf u_{j,i}, U_j\}$ as follows
\begin{equation}
\label{eq:distribution_Uj}
P( \bigllbracket{\boldsymbol{ \sigma}; U_j} ) \propto \prod^{d-t}_{j=1}\nolimits\prod^R_{i=1}\nolimits\boldsymbol{1}_{\stmanifold{n_j}{1}  }(\mathbf u_{j,i}) \cdot \prod^d_{j=d-t+1}\nolimits\boldsymbol{1}_{\stmanifold{n_j}{R}  }(U_j).
\end{equation} 
On the other hand, in the presence of Cauchy noise, the probability of the observed data tensor $\mathcal A$ conditioned on $\bigllbracket{\boldsymbol{ \sigma}; U_j}$ is given by 
\begin{equation}
\label{eq:distribution_given_Uj_know_A}
P\left(\mathcal A_{i_1\cdots i_d} \mid \bigllbracket{\boldsymbol{ \sigma}; U_j}_{i_1\cdots i_d}  \right) \propto \frac{1}{1 + \bigxiaokuohao{ \bigllbracket{\boldsymbol{ \sigma};U_j}_{i_1\cdots i_d} -\mathcal A_{i_1\cdots i_d}    }^2/\delta^2   },~1\leq i_j\leq n_j,~1\leq j\leq d.
\end{equation}
With \eqref{eq:distribution_Uj} and \eqref{eq:distribution_given_Uj_know_A} at hand, using Bayes's rule, the MAP estimation 
is given by
\begin{eqnarray*}
\{\boldsymbol{ \sigma}^*,U_j^* \} &=& \arg\max P\bigxiaokuohao{\bigllbracket{\boldsymbol{ \sigma};U_j }\mid \mathcal A  }\\
&=&\arg\max \frac{P\bigxiaokuohao{\mathcal A\mid \bigllbracket{\boldsymbol{ \sigma};U_j}}\cdot P\bigxiaokuohao{\bigllbracket{\boldsymbol{ \sigma};U_j}}  }{P(\mathcal A) }\nonumber\\
&=& \arg\max  \prod^{n_1,\ldots,n_d}_{i_1=1,\ldots,i_d=1}P\bigxiaokuohao{\mathcal A_{i_1\cdots i_d}\mid \bigllbracket{ \boldsymbol{ \sigma};U_j}_{i_1\cdots i_d}   } \cdot P\bigxiaokuohao{\bigllbracket{\boldsymbol{ \sigma};U_j}}\nonumber\\
&\overset{t\leftarrow -\log(t)}{=}& \arg\min \sum^{n_1,\ldots,n_d}_{i_1=1,\ldots,i_d=1}\log\bigxiaokuohao{ 1+ \bigxiaokuohao{ \bigllbracket{\boldsymbol{ \sigma};U_j}_{i_1\cdots i_d}-\mathcal A_{i_1\cdots i_d}  }^2/\delta^2    }\nonumber\\
&&~~~~~~~~~~~~~~~~~~~~ - \sum^{d-t}_{j=1}\sum^R_{i=1}\log\bigxiaokuohao{ \boldsymbol{1}_{\stmanifold{n_j}{1}  }(\mathbf u_{j,i})   } - \sum^d_{j=d-t+1} \log\bigxiaokuohao{  \boldsymbol{1}_{\stmanifold{n_j}{R}  }(U_j)  }\nonumber\\
&=& \arg\min \sum^{n_1,\ldots,n_d}_{i_1=1,\ldots,i_d=1}\log\bigxiaokuohao{ 1+ \bigxiaokuohao{ \bigllbracket{\boldsymbol{ \sigma};U_j}_{i_1\cdots i_d}-\mathcal A_{i_1\cdots i_d}  }^2/\delta^2    } \nonumber\\
&&~~~~~~~~~~~~~~~~~~~~ + \sum^{d-t}_{j=1}\sum^R_{i=1} \iota_{\stmanifold{n_j}{1}  }(\mathbf u_{j,i})    + \sum^d_{j=d-t+1}   \iota_{\stmanifold{n_j}{R}  }(U_j),    
\end{eqnarray*}
where in the last equality, we have defined $\log(0)=-\infty$. Therefore, from the above deduction,   to deal with \eqref{prob:cpd} in the presence of Cauchy noise (or even other heavy-tailed noise or outliers), we prefer to solve the following optimization model
\begin{equation} \label{prob:robust_orth_main}
\setlength\abovedisplayskip{4pt}
\setlength\abovedisplayshortskip{4pt}
\setlength\belowdisplayskip{4pt}
\setlength\belowdisplayshortskip{4pt}
\begin{split}
&\min~     \boldsymbol{\Phi}_\delta( \mathcal A-\llbracket \boldsymbol{ \sigma}; U_j\rrbracket     ) :=\frac{\delta^2}{2}\sum^{n_1,\ldots,n_d}_{i_1=1,\ldots,i_d=1}  \log \bigxiaokuohao{  1+  \bigxiaokuohao{ \bigllbracket{\boldsymbol{ \sigma};U_j}_{i_1\cdots i_d} -\mathcal A_{i_1\cdots i_d}    }^2/\delta^2   } \\
&~~    {\rm s.t.}~  \mathbf u_{j,i}\in \stmanifold{n_j}{1}, 1\leq j\leq d-t ,  1\leq i \leq R,\\
&~~~~~~~ U_j \in \stmanifold{n_j}{R},  d-t+1\leq j\leq d.
\end{split}
\end{equation}
Comparing \eqref{prob:robust_orth_main} with \eqref{prob:obj_ls}, we see that the difference is that the least squares loss $\ell_2(t) = t^2/2$ is replaced   by the statistically motivated loss function
\begin{equation}\label{loss:cauchy}
\phi_\delta(t) : = \frac{\delta^2}{2} \log\bigxiaokuohao{ 1+  { t^2  }/{\delta^2}     }.
\end{equation}
$\phi_\delta(\cdot)$ is called the Cauchy loss. In recent years,   various research has been focused on Cauchy loss based models; see, e.g., \cite{he2010maximum,sciacchitano2015variational,yang2015robust,guan2017truncated,li2018robust,mei2018cauchy,meng2020cauchy,kim2020cauchy}.

 We discuss some   properties of the proposed model \eqref{prob:robust_orth_main} from the robust statistics viewpoint, which shows \eqref{prob:robust_orth_main} is not only resistant to  Cauchy noise, but  may also be resistant to other heavy-tailed noise or outliers.    Firstly, we observe that
 \begin{equation}\label{loss:cauchy_derivative}
 \lim_{|t|\rightarrow+\infty} \phi_\delta'(t) = \lim_{|t|\rightarrow+\infty} \frac{t}{1+t^2/\delta^2} = 0.
 \end{equation}
 Such a property is called the redescending property in robust statistics \cite{huber2004robust}, and the minimizer of \eqref{prob:robust_orth_main} is called a redescending M-estimator. It is known that the redescending M-estimator is robust to heavy-tailed noise and outliers \cite{huber2004robust}. As a comparison, the derivative of the least squares loss $\ell_2(t)=t^2/2$ is $t$, whose limit is infinity, which does not have the redescending property.  Other loss functions admitting the redescending property include the Welsch loss \cite{holland1977robust,feng2015learning,feng2020statistical}, the Tukey loss \cite{beaton1974fitting}, the German loss \cite{ganan1985bayesian}, and so on.
 
 Secondly, the parameter $\delta$ in \eqref{loss:cauchy} controls the robustness of the model \eqref{prob:robust_orth_main}. 
 From \eqref{loss:cauchy_derivative}, we see that  the smaller $\delta$ is, the faster $\phi_\delta'(t) $ converges to zero. We plot $\phi_\delta^\prime(t)$ with different $\delta$ in the right panel of Fig. \ref{fig:loss}.  On the other hand, taking Taylor expansion of $\phi_\delta(t)$ at $0$ yields $\phi_\delta(t) = t^2/2 + o( t^2/\delta^2  )$, which shows that $\phi_\delta(t)\approx t^2/2$ as $\delta\rightarrow \infty$.
These observations imply that a small $\delta$ can enhance the robustness of \eqref{prob:robust_orth_main}. This also reminds us that our model \eqref{prob:robust_orth_main} is also resistant to Gaussian noise by simply setting a large enough $\delta$.     We   also plot $\phi_\delta(t)$ with different $\delta$ in the left panel of Fig. \ref{fig:loss}.

 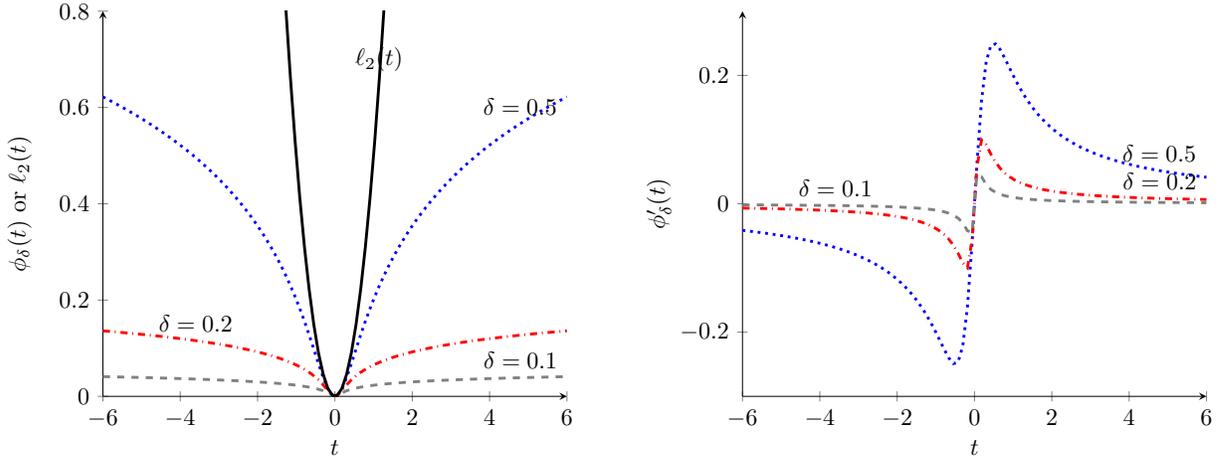
\begin{figure} 
	\centering

	\begin{tabular}{cc}
	\newcommand{\sigmaa}{0.5}
\newcommand{\sigmac}{0.2}
\newcommand{\sigmab}{0.1}
\begin{tikzpicture}[scale=0.9,/pgf/declare function={
	f=\sigmaa^2/2*ln(1+ x^2/\sigmaa^2);
	g=\sigmab^2/2*ln(1+ x^2/\sigmab^2);
	h=\sigmac^2/2*ln(1+ x^2/\sigmac^2);
	ls = x^2/2;
}]
\begin{axis}[
domain=-6:6,
samples=100,
axis lines = left,
ymin=0,
ymax=0.8,
xlabel={$t$},
ylabel={$\phi_\delta(t) $ or $\ell_2(t)$}
]
\addplot[dotted,very thick,color=blue] {f}node[pos = 0.9, 
anchor=south] {{\color{black} $\delta = \sigmaa$}};
\addplot[dashdotted,very thick,color=red] {h}node[pos = 0.2, 
anchor=south] {{\color{black} $\delta = \sigmac$}};
\addplot[dashed,very thick,color=gray] {g} node[pos = 0.9, 
anchor=south] {{\color{black} $\delta = \sigmab$}};
\addplot[ very thick,color=black] {ls} node[pos = 0.535, 
anchor=south] {{\color{black} $\ell_2(t)$}};
\end{axis}
\end{tikzpicture} 
&
	\newcommand{\sigmaa}{0.5}
\newcommand{\sigmac}{0.2}
\newcommand{\sigmab}{0.1}
\begin{tikzpicture}[scale=0.9,/pgf/declare function={
	f=x/(1+x^2/\sigmaa^2);
	g=x/(1+x^2/\sigmab^2);
	h=x/(1+x^2/\sigmac^2);
}]
\begin{axis}[
domain=-6:6,
samples=100,
axis lines = left,
ymin=-0.3,
ymax=0.3,
xlabel={$t$},
ylabel={$\phi_\delta^\prime(t)$}
]
\addplot[dotted,very thick,color=blue] {f}node[pos = 0.9, 
anchor=south] {{\color{black} $\delta = \sigmaa$}};
\addplot[dashdotted,very thick,color=red] {h}node[pos = 0.9, 
anchor=south] {{\color{black} $\delta = \sigmac$}};
\addplot[dashed,very thick,color=gray] {g} node[pos = 0.2, 
anchor=south] {{\color{black} $\delta = \sigmab$}};
\end{axis}
\end{tikzpicture} 
\end{tabular}
 	\caption{\small Left: Plots of  $\phi_\delta(t)=
	\frac{\delta^2}{2}\log\bigxiaokuohao{1+ t^2/\delta^2}$ with   different $\delta$ values versus $\ell_2(t) = t^2/2$; Right: Plots of $\phi_\delta^\prime(t)=t/(1+t^2/\delta^2)$. 
	$\sigma = 0.1$ (the dashed curve), $\sigma = 0.2$
	(the dotted-dashed curve), and  $\sigma = 0.5$
	(the dotted curve); $\ell_2(t)$ (the solid curve).}
\label{fig:loss}
\end{figure}

 \begin{remark}
We discuss several differences between our model \eqref{prob:robust_orth_main} and some existing robust tensor models. In recent years, robust techniques have been incorporated into tensor decomposition/approximation/recovery/completion/PCA problems, where the $L_1$ loss function, namely, $\ell_1(t)=|t|$, is frequently employed to deliver   robustness. In general, such kind of models can be formulated as \cite{goldfarb2014robust}
\begin{equation}
\label{prob:l1_tensor_problem}
\min_{\mathcal X\in\T} \bignorm{\boldsymbol{ L}(\mathcal X)-\mathbf b }_1 + \lambda R(\mathcal X),
\end{equation}
where $\boldsymbol{ L}$ is a linear operator,  and $\mathbf b$ has the same size as $\boldsymbol{ L}(\mathcal X)$;  $R(\mathcal X)$ denotes a certain regularizer that controls the low-rankness of $\mathcal X$, such as the sum of nuclear norms of unfolding matrices of $\mathcal X$ \cite{signoretto2013learning}, and $\lambda>0$ is the regularization parameter.  A special case of \eqref{prob:l1_tensor_problem} is the robust tensor PCA, in which $\boldsymbol{ L}$ is the identity operator and $\mathbf b$ denotes the observed tensor \cite{goldfarb2014robust}. 
It is known that $L_1$ loss is more suitable for Laplacian noise; on the other hand, one sees that the derivative of $|t|$ does not tend to zero as $|t|\rightarrow+\infty$, meaning that it does not admit the redescending property, while it was pointed out in \cite{maronna1979bias} that the $L_1$ estimator might behave as bad as the $\ell_2(t)$ estimator in some cases. Comparing with the resulting tensor, \eqref{prob:l1_tensor_problem} yields a full tensor of size $n_1\times\cdots\times n_d$, while ours is compressed into a set of factor matrices, which   takes much less storage. Moreover, our orthonormality assumption on some factor matrices is more suitable for certain applications \cite{de2011short,de2010algebraic,sidiropoulos2000blind,de2013joint,sorensen2010parafac,shashua2001linear}.

In \cite{anandkumar2016tensor}, a robust tensor CP decomposition model has been considered. The differences are that the noise there are required to be sparse, and all the factor matrices are assumed to be columnwisely orthogonal, which are stringent. By using outlier detection techniques, \cite{pravdova2001robust} proposed a robust Tucker model. However, the underlying model cannot be   clearly formulated as an optimization problem, and the tensor model is different from ours. By using variational inference and     Kullback-Leibler  divergence, \cite{cheng2016probabilistic} devised a robust algorithm to find CP approximation with orthonormal factors, where the model and the solution method are quite different from ours.  In particular, the authors pointed out that their algorithm boils down to the alternating least squares \cite{sorensen2012canonical} in the absence of outliers. 
In a recent survey \cite{hong2020generalized}, various statistically motivated loss functions are incorporated into tensor CPD, in which the Huber's loss is considered. As Huber's loss can be regarded as a smoothed $\ell_1$ loss, it does not admit the redescending property as well. The orthonormality is   not taken into account in  \cite{hong2020generalized}.  Note that the idea of employing Cauchy loss has been considered in the authors' earlier work   \cite{yang2015robust}. Comparing with \eqref{prob:robust_orth_main}, the resulting tensor in \cite{yang2015robust} is   a full tensor and also does not take into account the orthonormality, and the solution method is also different.

\end{remark}

The remaining   problem is how to solve \eqref{prob:robust_orth_main} efficiently. For this purpose, several quantitative properties concerning the Cauchy loss for designing and analyzing the solution method are first introduced in the following subsection.    
 \subsection{Quantitative properties concerning $\phi_\delta(\cdot)$}
First, 
 we introduce the so-called half-quadratic (HQ) property of $\phi_\delta(\cdot)$, which turns the function into a weighted least squares problem and is crucial for designing the algorithm. Such a property of the Cauchy loss has appeared in the literature; see, e.g., \cite{he2010maximum,guan2017truncated}, in which the verification  is based on the utilization of conjugate functions. While   we present a very direct and concise proof.    Recall that we have defined $\log(0) = -\infty$.
 \begin{lemma}[Half-quadratic property]\label{lem:hq}
 Given $|t|<+\infty$, it holds that
 	\begin{align}\label{auxiliary_lemma::eq} 
 	\phi_\delta(t)= \min_{\omega\geq 0} \frac{\omega}{2} t^2 +\frac{\delta^2}{2}\varrho(\omega),
 	\end{align}
 	where 
 	$ 
 	\varrho(\omega)=\omega- \log(\omega)-1.
 	$
 	Moreover, the minimizer of \eqref{auxiliary_lemma::eq} is given by
 	\begin{equation}\label{eq:sec:alg:3}
 	\omega^*=\frac{\delta^2}{\delta^2+t^2}.
 	\end{equation}
 \end{lemma}
 \begin{proof}
First we  verify that \eqref{eq:sec:alg:3} is a minimizer of the right hand-side of \eqref{auxiliary_lemma::eq}.	Denote $g(\omega):= \omega t^2/2 + \delta^2\varrho(\omega)/2$.	As $\varrho(\cdot)$ is convex, it suffices to show that   $\omega^*$ in \eqref{eq:sec:alg:3} is a stationary point of  $\inf_{\omega\geq 0}g(\omega)$. Since $|t|<+\infty$, we see that the minimizer of $\inf_{\omega\geq 0}g(\omega)$ cannot occur at $\omega = 0$. Thus any stationry point of $\inf_{\omega\geq 0}g(\omega)$ meets 
 	$$g'(\omega) = 0 \Leftrightarrow t^2+\delta^2 -\frac{\delta^2}{\omega}=0,$$
 	and so $\omega = ( 1+ t^2/\delta^2)^{-1}$, which is exactly \eqref{eq:sec:alg:3}. Inserting this expression into \eqref{auxiliary_lemma::eq},  we get
 	\begin{eqnarray*}
 		2g(t) &=& \frac{\delta^2}{\delta^2+t^2}(t^2+\delta^2)  + \delta^2\log ( 1+ t^2/\delta^2  ) - \delta^2\\
 		&=&  \delta^2\log ( 1+ t^2/\delta^2  ),
 	\end{eqnarray*}
 	boiling down to the expression of $\phi_\delta(t)$. The proof is completed.
 \end{proof}
 Note that the HQ property has a very clear indication on robustness: Take $t=  \mathcal A_{i_1\cdots i_d}-\bigllbracket{\boldsymbol{ \sigma};U_j}_{i_1\cdots i_d}$ in Lemma \ref{lem:hq} as the noise; we see that the larger   the magnitude of $t$, the smaller the weight $\omega$ it yields, and so the corresponding $\phi_\delta(t)$ is less important in the objective $\boldsymbol{ \Phi}_\delta(\cdot)$ in \eqref{prob:robust_orth_main}.
 
The next two properties are helpful for convergence analysis. Recalling that $\phi_\delta'(t) = \frac{\delta^2t}{\delta^2+t^2} $, we have		
 \begin{proposition}[Lipschitz gradient] For any $t_1,t_2\in\mathbb R$ and $\delta>0$, it holds that
 	\label{prop:lipschitz_gradient} 
 	\[
 	\bigjueduizhi{ \frac{\delta^2t_1}{\delta^2 + t_1^2}  - \frac{\delta^2t_2}{\delta^2 + t_2^2} } \leq |t_1-t_2|.
 	\]
 \end{proposition}
\begin{proof}
By the mean value theorem, 	It suffices to show that
$
|\phi_\delta''(t)| \leq 1.
$ In fact,   
\[
|\phi^{\prime\prime}_\delta(t)| = \bigjueduizhi{ \frac{\delta^2 (\delta^2 - t^2)}{ (\delta^2 + t^2)^2  }    } \leq \bigjueduizhi{ \frac{ \delta^2 }{ \delta^2 + t^2 }   } \leq 1,
\]
and the result follows.
\end{proof}

\begin{proposition}[Liptshitz-like inequality] 	\label{prop:E2}
Let $t_1,t_2 \in\mathbb R$ be arbitrary, and let $\delta>0$. Then it holds  that
	$$ |e|:=\bigjueduizhi{ \delta^2 t_{1}  \left(   \frac{1}{ \delta^2 + t_1^2  } - \frac{1}{\delta^2 +  t_2^2  } \right) } \leq \bigjueduizhi{t_1-t_2}.$$ 
\end{proposition}
\begin{proof}
	It is clear that
	\begin{equation*}
 |e| =\left|   \sigma^2t_1 \frac{ (t_1 + t_2) (  t_1 - t_2  )   }{ (\sigma^2 + t_1^2  )  ( \sigma^2 + t_2^2   )  } \right| \leq  \sigma^2|t_1| \frac{ |t_1|+|t_2|  }{   (\sigma^2 + t_1^2  )  ( \sigma^2 + t_2^2   )          } \cdot | t_1-t_2 |.
	\end{equation*}
	To prove the above relation, it suffices to show the coefficient of $|t_1-t_2|$ is not greater than $1$, i.e.,
	\[
	\varphi(t_1,t_2)  := (\sigma^2 + t_1^2  )  ( \sigma^2 + t_2^2   ) - \sigma^2|t_1|( |t_1|+|t_2| ) \geq 0.
	\]
	In fact, 
	\begin{eqnarray*}
		\varphi(t_1,t_2)  &=& \sigma^4 + \sigma^2t_2^2 + |t_1t_2| ( |t_1t_2| - \sigma^2   )  \\
		&\geq& \sigma^4 + \sigma^2t_2^2 - \frac{\sigma^4}{4} \geq 0.
	\end{eqnarray*}
	Therefore, $|e|\leq |t_1-t_2|$, as desired. 
\end{proof}

\section{HQ-ADMM}\label{sec:alg}
By using Lemma \ref{lem:hq},  we equivalently rewrite the objective function $\boldsymbol{ \Phi}_\delta(\cdot)$ of \eqref{prob:robust_orth_main} in what follows. Specifically,
since $\boldsymbol{ \Phi}_\delta(\cdot)$ is the sum of $\phi_\delta(\cdot)$ functions,   taking $t =  \mathcal A_{i_1\cdots i_d} - \llbracket \boldsymbol{ \sigma}; U_j\rrbracket _{i_1\cdots i_d}$ in Lemma \ref{lem:hq}, we have
\begin{eqnarray}\label{eq:cost}
&&\boldsymbol{\Phi}_\delta( \mathcal A-\llbracket \boldsymbol{ \sigma}; U_j\rrbracket     ) \nonumber\\
=&&  \frac{1}{2} \min_{ \mathcal W_{i_1\cdots i_d} \geq 0} \sum^{n_1,\ldots,n_d}_{i_1=1,\ldots,i_d}  \bigzhongkuohao{ \mathcal W_{i_1\cdots i_d} \bigxiaokuohao{\mathcal A_{i_1\cdots i_d} - \llbracket \boldsymbol{ \sigma}; U_j\rrbracket _{i_1\cdots i_d}  }^2 + \delta^2\varrho(\mathcal W_{i_1\cdots i_d})  },  \label{eq:hq:1}
\end{eqnarray}
where we denote $\mathcal W = (\mathcal W_{i_1\cdots i_d})\in \T$ as a tensor variable. From Lemma \ref{lem:hq}, we see that the optimizer is $\mathcal W_{i_1\cdots i_d}= \delta^2\bigxiaokuohao{1+ \bigxiaokuohao{ \bigllbracket{\boldsymbol{ \sigma}; U_j}_{i_1\cdots i_d}-\mathcal A_{i_1\cdots i_d}   }^2/\delta^2  }^{-1}$. As explained in the paragraph below Lemma \ref{lem:hq},   $\mathcal W$ can be interpreted as weights to the problem.   From the expression of $\mathcal W$, we see that the larger the noise is, the smaller the weight gives to the problem. Such a mechanism helps   mitigate  heavy-tailed noise or outliers.

In view of \eqref{eq:cost}, a straightforward idea to solve \eqref{prob:robust_orth_main} (with the objective replaced by \eqref{eq:cost}) is to employing an alternating minimization method (AM)  by iteratively updating $\boldsymbol{ \sigma},U_j$, and $\mathcal W$. In fact,  applying AM to solve Cauchy loss-based problems have been considered in the literature; see, e,g., \cite{he2010maximum,guan2017truncated}. However, for our problem, this would result in that the subproblems related to $U_j$ do not have closed-form solutions. \cite{li2018robust} also applied AM to solve Cauchy loss-based problem; however, as their proposed model is unconstrained and the objective function is smooth, AM yields closed-form solutions to each subproblem. \cite{sciacchitano2015variational}   incorporated Cauchy loss into models for image processing. However, the problem is convexified by imposing a quadratic term, which results in that the Cauchy loss related subproblem admits a unique solution that can be analytically solved by solving a cubic equation. If the subproblem is nonconvex, then numerical methods have to be applied to solving the Cauchy loss related subproblem, as pointed out in \cite{sciacchitano2015variational}, which might result in inefficiency. For   other Cauchy loss based image processing problems, \cite{mei2018cauchy,meng2020cauchy,kim2020cauchy} proposed to use the conventional alternating direction method of multipliers (ADMM) directly. However, without noticing the HQ property, in the ADMM, solving the Cauchy loss related subproblem also does not admit a closed-form solution. As a result, solving such a subproblem  still requires an iterative method. \cite{yang2015robust} used a linearization technique, which ignored the HQ property. 

In view of the above limitations in dealing with Cauchy loss-based problems, in this section, by combining the HQ property and   the ADMM framework, we proposed a new method, termed as   HQ-ADMM, to solve our model \eqref{prob:robust_orth_main}. The advantage of HQ-ADMM is that all the subproblems involved in the algorithm admit closed-form solutions. In what follows, we derive our method step by step.

Note   that \eqref{eq:cost} is quadratic with respect to each $U_j$, leading to the following formulation 
\begin{equation*}
 \boldsymbol{\Phi}_\delta( \mathcal A-\llbracket \boldsymbol{ \sigma}; U_j\rrbracket     )  = \frac{1}{2} \min_{\mathcal W_{i_1\cdots i_d}\geq 0  }\normF{ \sqrt \mathcal W  \circledast \bigxiaokuohao{ \mathcal A -  \llbracket \boldsymbol{ \sigma}; U_j\rrbracket }   }^2 + \frac{\delta^2}{2}\sum^{n_1,\ldots,n_d}_{i_1=1,\ldots,i_d}   \varrho(\mathcal W_{i_1\cdots i_d})    ,  \label{eq:hq:2}
\end{equation*}
where $\sqrt \mathcal W = (\sqrt {\mathcal W_{i_1\cdots i_d}}) \in \T$ and `$\circledast$' denotes the Hadamard product. With this expression at hand,  by introducing a slack variable $\mathcal T \in \T$, we rewrite     \eqref{prob:robust_orth_main} as
\begin{equation} \label{prob:robust_orth_main_P}
\setlength\abovedisplayskip{4pt}
\setlength\abovedisplayshortskip{4pt}
\setlength\belowdisplayskip{4pt}
\setlength\belowdisplayshortskip{4pt}
\begin{split}
&\min_{\boldsymbol{ \sigma},U_j,\mathcal T,\mathcal W}~     \boldsymbol{\Phi}_\delta( \mathcal A-\mathcal T    ) = \frac{1}{2}  \normF{ \sqrt \mathcal W  \circledast\bigxiaokuohao{ \mathcal A -  \mathcal T }   }^2 + \frac{\delta^2}{2}\sum^{n_1,\ldots,n_d}_{i_1=1,\ldots,i_d}   \varrho(\mathcal W_{i_1\cdots i_d})   \\
&~~~~\,    {\rm s.t.}~  \mathcal T = \llbracket \boldsymbol{ \sigma}; U_j\rrbracket,~\mathcal W\geq 0,\\
&~~~~~~~~~  \mathbf u_{j,i}^\top \mathbf u_{j,i} =1, 1\leq j\leq d-t ,  1\leq i \leq R,\\
& ~~~~~~~~~ U_j^\top U_j = I,  d-t+1\leq j\leq d.
\end{split}
\end{equation}
By introducing a Lagrangian multiplier $\mathcal Y\in\T$, the augmented Lagrangian function of \eqref{prob:robust_orth_main_P} is given by
\begin{eqnarray}\label{eq:aug_L}
&&L_\tau(\boldsymbol{ \sigma},U_j,\mathcal T, \mathcal Y,\mathcal W  ):= \frac{1}{2}  \normF{ \sqrt \mathcal W  \circledast \bigxiaokuohao{ \mathcal A -  \mathcal T }   }^2 + \frac{\delta^2}{2}\sum^{n_1,\ldots,n_d}_{i_1=1,\ldots,i_d}    \varrho(\mathcal W_{i_1\cdots i_d}) \nonumber \\
&&~~~~~~~~~~~~~~~~~~~~~~~~~~~~ ~ - \innerprod{\mathcal Y}{\llbracket \boldsymbol{ \sigma}; U_j  \rrbracket - \mathcal T} + \frac{\tau}{2}\normF{\llbracket \boldsymbol{ \sigma}; U_j\rrbracket - \mathcal T   }^2 ,
\end{eqnarray}
where $\tau>0$.  In what follows, for notational convenience we denote     $(\mathcal Y + \tau\mathcal T) \bigotimes_{l\neq j}^d \mathbf u_{l,i} \in\mathbb R^{n_j}$ as the gradient of $\innerprod{\mathcal Y + \tau\mathcal T}{\bigotimes^d_{l=1}\mathbf u_{l,i}}$ with respect to $\mathbf u_{j,i}$.
Then, the last two terms of \eqref{eq:aug_L} can be rewritten as
\begin{eqnarray}
- \innerprod{\mathcal Y}{\llbracket \boldsymbol{ \sigma}; U_j  \rrbracket - \mathcal T} + \frac{\tau}{2}\normF{\llbracket \boldsymbol{ \sigma}; U_j\rrbracket - \mathcal T   }^2 &=& \innerprod{\mathcal Y}{\mathcal T} + \frac{\tau}{2}\normF{\mathcal T}^2 - \innerprod{\mathcal Y + \tau \mathcal T}{\tensorsigmaU} + \frac{\tau}{2}\boldsymbol{ \sigma}^\top\boldsymbol{ \sigma} \label{eq:hq:3}\\
&=&  \innerprod{\mathcal Y}{\mathcal T} +\frac{\tau}{2}\normF{\mathcal T}^2 - \innerprod{\mathcal Y + \tau \mathcal T}{\sum^R_{i=1}\sigma_i \bigotimes_{j=1}^d\mathbf u_{j,i}  } + \frac{\tau}{2}\boldsymbol{ \sigma}^\top\boldsymbol{ \sigma}\nonumber\\
&=& \innerprod{\mathcal Y}{\mathcal T} +\frac{\tau}{2}\normF{\mathcal T}^2 - \sum^R_{i=1}\sigma_i \innerprod{ (\mathcal Y + \tau\mathcal T) \bigotimes_{l\neq j}^d \mathbf u_{l,i}   }{\mathbf u_{j,i}} + \frac{\tau}{2}\boldsymbol{ \sigma}^\top\boldsymbol{ \sigma},\nonumber
\end{eqnarray}
where the first equality is due to Proposition \ref{prop:orth}. 

Before presenting the algorithm, we first derive the stationary point system. To this end, we further define Lagrangian multipliers $\eta_{j,i}\in\mathbb R$, $ 1\leq j\leq d-t$, $ 1\leq i\leq R$ attached to the constraints $\mathbf u_{j,i}^\top\mathbf u_{j,i}=1$, and $\Lambda_j\in\mathbb R^{R\times R}$, $ d-t+1\leq j\leq d$ attached to $U^\top_j U_j=I$, where $\Lambda_j$'s are symmetric matrices.  Denote 
\begin{eqnarray}
\label{eq:aug_LL}
&&\hat L_\tau(\boldsymbol{ \sigma},U_j,\mathcal T, \mathcal Y,\mathcal W  ) := L_\tau(\boldsymbol{ \sigma},U_1,\ldots,U_d,\mathcal T, \mathcal Y,\mathcal W  ) \nonumber\\
&&~~~~~~~~~~~~~~~~~~~~~~~~~~~ +  \sum_{j,i=1}^{d-t,R}\nolimits \eta_{j.i}\bigxiaokuohao{ \mathbf u_{j,i}^\top \mathbf u_{j,i} -1} + \sum^{d}_{j=d-t+1}\nolimits\innerprod{\Lambda_j}{ U_j^\top U_j - I}.
\end{eqnarray} 
 Thus taking derivative of $\hat L(\cdot)$ with respect to each $\mathbf u_{j,i}, 1\leq j\leq d-t, 1\leq i\leq R$ and noticing \eqref{eq:hq:3} yields
\begin{equation}
\label{eq:stat:u}
{\sigma_i (\mathcal Y + \tau\mathcal T)\bigotimes_{l\neq j}^d\nolimits\mathbf u_{l,i} =     \eta_{j,i} \mathbf u_{j,i}   }, ~  1  \leq j\leq d-t,~ 1\leq i\leq R.
\end{equation}
Since $\mathbf u_{j,i}$'s are normalized, we get $\eta_{j,i} = \sigma_i\innerprod{\mathcal Y + \tau\mathcal T}{\bigotimes^d_{l=1}\mathbf u_{l,i}}$. On the other hand, noticing the representation \eqref{eq:hq:3}, taking derivative of $\hat L(\cdot)$ with respect to $\boldsymbol{ \sigma}$ gives that $\sigma_i = \innerprod{\mathcal Y + \tau\mathcal T}{\bigotimes^d_{l=1}\mathbf u_{l,i}}/\tau$, which together with the expression of $\eta_{j,i}$ gives $\eta_{j,i} = \sigma_i^2\tau$; therefore, \eqref{eq:stat:u} is in fact as follow
	\begin{equation}
	\label{eq:stat:u_correct}
	{ (\mathcal Y + \tau\mathcal T)\bigotimes_{l\neq j}^d\nolimits\mathbf u_{l,i} =    \sigma_i\tau\mathbf u_{j,i}   }, ~  1  \leq j\leq d-t,~ 1\leq i\leq R.
	\end{equation}
	Next, taking derivative with respect to $\mathbf u_{j,i}, d-t+1\leq j\leq d, 1\leq i\leq R$ and noticing \eqref{eq:hq:3} gives
	\begin{equation}
	\label{eq:stat:u_orth}
		\sigma_i(\mathcal Y + \tau\mathcal T)\bigotimes_{l\neq j}^d\nolimits\mathbf u_{l,i} =   \sum^R_{r=1}  \nolimits(\Lambda_j)_{i,r} \mathbf u_{j,r }, ~  1  \leq j\leq d-t,~ 1\leq i\leq R.\\
	\end{equation}
Denote $\mathcal E\in\T$ as the all-one tensor;	taking derivative with respect to $\mathcal T$ and rearranging terms yields
	\begin{eqnarray}
	\label{eq:stat:P}
	&& \mathcal W\circledast\bigxiaokuohao{\mathcal T - \mathcal A} + \mathcal Y - \tau\bigxiaokuohao{\llbracket\boldsymbol{ \sigma};U_j\rrbracket -\mathcal T} = 0\nonumber\\
\Leftrightarrow&&	\bigxiaokuohao{ \mathcal W + \tau \mathcal E }\circledast \mathcal T =  \mathcal W\circledast \mathcal A - \mathcal Y + \tau\tensorsigmaU.
	\end{eqnarray}
	As a result, taking \eqref{eq:stat:u_correct}, \eqref{eq:stat:u_orth}, \eqref{eq:stat:P} and Lemma \ref{lem:hq} into account, any stationary point $\{\boldsymbol{ \sigma},U_j,\mathcal T,\mathcal Y,\mathcal W   \}$ satisfies the following system
\begin{equation}  \label{eq:kkt}  \footnotesize
\left\{  
\begin{array}{lr}
	{ (\mathcal Y + \tau\mathcal T)\bigotimes_{l\neq j}^d\mathbf u_{l,i} =    \sigma_i\tau\mathbf u_{j,i}   }, & ~  1  \leq j\leq d-t,~ 1\leq i\leq R,\\  
	\mathbf u_{j,i}^\top \mathbf u_{j,i} =1, &~  1  \leq j\leq d-t,~ 1\leq i\leq R,\\
		\sigma_i(\mathcal Y + \tau\mathcal T)\bigotimes_{l\neq j}^d\mathbf u_{l,i} =   \sum^R_{r=1}  (\Lambda_j)_{i,r} \mathbf u_{j,r }, & ~  1  \leq j\leq d-t,~ 1\leq i\leq R,\\
			U_j^\top U_j = I, &d-t+1 \leq j\leq d,\\
				\bigxiaokuohao{ \mathcal W + \tau \mathcal E }\circledast \mathcal T =  \mathcal W\circledast \mathcal A - \mathcal Y + \tau\llbracket\boldsymbol{ \sigma};U_j\rrbracket,\\
\llbracket\boldsymbol{ \sigma};U_j\rrbracket = \mathcal T,\\
\mathcal W_{i_1\cdots i_d} =  {\delta^{2}}\bigxiaokuohao{ \delta^2 + \bigxiaokuohao{\mathcal T_{i_1\cdots i_d} - \mathcal A_{i_1\cdots i_d}   }^2   }^{-1} .\\  
\end{array}  
\right.
\end{equation}

\paragraph{HQ-ADMM framework} Combining the HQ property and the ADMM, our HQ-ADMM computes the following subproblems at each iterate
\begin{equation*}  \label{alg:hq-admm_framework} \footnotesize
\left\{  
\begin{array}{lr}
 U^{k+1}_{j} \in \arg\min_{\|\mathbf u_{j,i}\|=1, 1\leq i\leq R} L_\tau(\boldsymbol{ \sigma}^k, U_1^{k+1},\ldots,U^{k+1}_{j-1},U_j,U^k_{j+1},\ldots,U^k_d,\mathcal T^k,\mathcal Y^k,\mathcal W^k   ) ,& 1 \leq j\leq d-t, \\
 U^{k+1}_j \in \arg\min_{ U_j^\top U_j = I }L_\tau(\boldsymbol{ \sigma}^k, U_1^{k+1},\ldots,U^{k+1}_{j-1},U_j,U^k_{j+1},\ldots,U^k_d,\mathcal T^k,\mathcal Y^k,\mathcal W^k   ) ,& d-t+1 \leq j\leq d,\\ 
\mathcal T^{k+1} = \arg\min_{\mathcal T} L_{\tau}(\boldsymbol{ \sigma}^{k},U^{k+1}_j,\mathcal T,\mathcal Y^k,\mathcal W^k ),\\
\mathcal Y^{k+1} = \mathcal Y^k - \tau\bigxiaokuohao{ \llbracket \boldsymbol{ \sigma}^{k}; U^{k+1}_j\rrbracket - \mathcal T^{k+1}  },\\
\boldsymbol{ \sigma}^{k+1} = \arg\min_{\boldsymbol{ \sigma}} L_\tau(\boldsymbol{ \sigma}, U^{k+1}_j, \mathcal T^{k+1},\mathcal Y^{k+1},\mathcal W^k), \\
\mathcal W^{k+1} = \arg\min_{\mathcal W} L_{\tau}(\boldsymbol{ \omega}^{k+1},U^{k+1}_j,\mathcal T^{k+1},\mathcal Y^{k+1},\mathcal W  ).
\end{array}  
\right.
\end{equation*}
Comparing with the standard ADMM framework,   HQ-ADMM involves an additional subproblem to update the weights $\mathcal W$.  In what follows, we present how to solve each subproblem. 

\subparagraph{$U_j$-subproblems}
For notational convenience, let $$\mathbf v_{j,i}^{k+1}:=(\mathcal Y^k+ \tau\mathcal T^k){\mathbf u_{1,i}^{k+1}\otimes \cdots\otimes \mathbf u_{j-1,i}^{k+1} \otimes \mathbf u_{j+1,i}^k \otimes\cdots\otimes \mathbf u_{d,i}^k  }$$ represent the gradient of $\innerprod{\mathcal Y^k + \tau\mathcal T^k}{\bigotimes^d_{l=1}\mathbf u_{l,i} }$ with respect to $\mathbf u_{j,i}$ at the point $(\mathbf u^{k+1}_{1,i},\ldots,\mathbf u^{k+1}_{j-1,i},\mathbf u^{k}_{j,i},\ldots,\mathbf u^k_{d,i}  )$. Denote $V^{k+1}_j := [\mathbf v^{k+1}_{j,1},\ldots,\mathbf v^{k+1}_{j,R} ] \in \mathbb R^{n_j\times R}$.

When $1\leq j\leq d-t$, from the definition of $L_\tau(\cdot)$, $\mathbf v_{j,i}$, and noticing the expression \eqref{eq:hq:3}, we have that each column of $U_j$ can be updated as follows
\begin{equation*}
\mathbf u^{k+1}_{j,i} = \arg\min_{\|\mathbf u_{j,i}\|=1 } -\sigma_i^k\innerprod{ \mathbf v^{k+1}_{j,i}   }{\mathbf u_{j,i}} ~\Leftrightarrow~ \mathbf u^{k+1}_{j,i} = \mathbf v^{k+1}_{j,i}/\|\mathbf v^{k+1}_{j,i}\|,~1\leq i\leq R.
\end{equation*}
However, for the convenience of convergence analysis we compute the following instead
\begin{equation}\label{prob:u_subproblem}
\mathbf u^{k+1}_{j,i} = \tilde{\mathbf v}^{k+1}_{j,i} /\|\tilde{\mathbf v}^{k+1}_{j,i}\|,~{\rm where}~\tilde{\mathbf v}_{j,i}^{k+1} = \sigma^k_i \mathbf v^{k+1}_{j,i} + \alpha \mathbf u^k_{j,i},~  1\leq i\leq R;
\end{equation}
here $\alpha>0$ is an arbitrary   constant. Note that $\mathbf u^{k+1}_{j,1},\ldots, \mathbf u^{k+1}_{j,R}$ can be updated simultaneously. 

When $d-t+1\leq j\leq d$, from the definition of $L_\tau$, $\mathbf v^{k+1}_{j,i}$, $V^{k+1}_j$ and recalling \eqref{eq:hq:3}, it follows
\begin{equation*}
U^{k+1}_j = \arg\min_{ U_j^\top U_j=I} - \sum^R_{i=1}\innerprod{ \sigma_i \mathbf v^{k+1}_{j,i}   }{\mathbf u_{j,i}} = \arg\max_{U_j^\top U_j=I} \innerprod{ V^{k+1}\cdot {\rm diag}(\boldsymbol{ \sigma}^k)  }{U_j},
\end{equation*}
where ${\rm diag}(\boldsymbol{ \sigma}) = {\rm diag}[\sigma_1,\ldots,\sigma_R ]\in\mathbb R^{R\times R}$ is a diagonal matrix. Similar to \eqref{prob:u_subproblem}, we in fact compute the following problem instead
\begin{equation}\label{prob:U_subproblem}
U^{k+1}_j =  \arg\max_{U_j^\top U_j=I} \innerprod{ \tilde{V}^{k+1}_j}{U_j},~{\rm where}~\tilde{V}^{k+1}_j = V^{k+1}\cdot {\rm diag}(\boldsymbol{ \sigma}^k) + \alpha U^k_j .
\end{equation}
The above problem is   to compute the polar decomposition of $\tilde{V}^{k+1}_j$, which admits a closed-form solution. Specifically, assume $\tilde{V}^{k+1}_j = P\Xi Q^\top$ is the SVD of $\tilde{V}^{k+1}_j$, where $P\in \mathbb R^{n_j\times R}$, $\Lambda,Q\in\mathbb R^{R\times R}$, $P^\top P = I$, $Q^\top Q = QQ^\top = I$, $\Xi = {\rm diag}(\lambda_1,\ldots, \lambda_R)$ with $\lambda_i$ being the singular value of $\tilde{V}^{k+1}_j$. Then $U^{k+1}_j = PQ^\top$. Moreover, letting $H^{k+1}_j:=Q\Xi Q^\top$. Then we   see that \eqref{prob:U_subproblem} gives the following relation
\begin{equation}\label{eq:U_subproblem_1}
\tilde{V}^{k+1}_j = U^{k+1}_j H^{k+1}_j.
\end{equation}

\subparagraph{$\mathcal T$-, $\boldsymbol{ \sigma}$- and $\mathcal W$-subproblems}
From 	\eqref{eq:stat:P}, we have that
\begin{equation}
\label{prob:P_subproblem}
\mathcal T^{k+1}_{i_1\cdots i_d} = \bigxiaokuohao{  \mathcal W^k_{i_1\cdots i_d}\mathcal A_{i_1\cdots i_d} - \mathcal Y^k_{i_1\cdots i_d}   + \tau   \bigllbracket{\boldsymbol{ \sigma}^k ; U^{k+1}_j}_{i_1\cdots i_d}     }/\bigxiaokuohao{ \mathcal W^{k}_{i_1\cdots i_d} + \tau    }.
\end{equation}
To compute $\boldsymbol{ \sigma}^{k+1}$, from the expression of \eqref{eq:hq:3} it is easily seen that
\begin{equation}
\label{prob:sigma_subproblem}
\sigma^{k+1}_i = (\mathcal Y^{k+1}+\tau \mathcal T^{k+1})\bigotimes^d_{j=1}\nolimits\mathbf u^{k+1}_{j,i} /\tau,~1\leq i\leq R.
\end{equation}
To compute $\mathcal W^{k+1}$, similar to \eqref{eq:kkt} we have
\begin{equation}\label{prob:W_subproblem}
\mathcal W^{k+1}_{i_1\cdots i_d} =  {\delta^{2}}\bigxiaokuohao{ \delta^2 + \bigxiaokuohao{\mathcal T^{k+1}_{i_1\cdots i_d} - \mathcal A_{i_1\cdots i_d}   }^2   }^{-1} .
\end{equation}

In summary, the HQ-ADMM is described in Algorithm \ref{alg:hq_admm}, where each subproblem admits a closed-form solution.
\begin{algorithm}[h] 
	\algsetup{linenosize=\tiny}
	\footnotesize
	\caption{HQ-ADMM for solving \eqref{prob:robust_orth_main}	\label{alg:hq_admm}}
	\begin{algorithmic}[1]
		\REQUIRE $U_j^0 = [ \mathbf u_{j,i}^0,\ldots,\mathbf u^0_{j,R}]$,  with $\|\mathbf u_{j,i}\|=1$, $1\leq j\leq d-t$, $1\leq i\leq R$; $(U^0)^\top U^0_j = I$, $d-t+1\leq j\leq d$;  $\boldsymbol{ \sigma}^0$,  $\mathcal T^0$, $\mathcal Y^0$, $\mathcal W^0$,  $\alpha>0$, $\tau>0$, $\delta>0$.
		\FOR{$k=0,1,\ldots,$}
		\STATE  Compute $\mathbf u^{k+1}_{j,i}$ via \eqref{prob:u_subproblem},~~~$1\leq j\leq d-t,1\leq i\leq R$
		\STATE  Compute $U^{k+1}_j$ via \eqref{prob:U_subproblem},~~~$d-t+1\leq j\leq d$ 
		\STATE Compute $\mathcal T^{k+1}$ via \eqref{prob:P_subproblem},
		\STATE Compute $\mathcal Y^{k+1} = \mathcal Y^k - \tau\bigxiaokuohao{ \llbracket \boldsymbol{ \sigma}^{k}; U^{k+1}_j\rrbracket - \mathcal T^{k+1}  }$,
				\STATE  Compute $\boldsymbol{ \sigma}^{k+1}$ via \eqref{prob:sigma_subproblem},
		\STATE  Compute $\mathcal W^{k+1}$ via \eqref{prob:W_subproblem}.
		\ENDFOR
	\end{algorithmic}
\end{algorithm}

\begin{remark}
	\label{rmk:hq_admm}
1. HQ-ADMM can be applied to a more general form of \eqref{prob:robust_orth_main}. Specifically, consider the data-fitting term given by $\boldsymbol{ \Phi}_\delta\bigxiaokuohao{\boldsymbol{L}\bigxiaokuohao{ \bigllbracket{\boldsymbol{ \sigma};U_j}  }-\mathbf b}$, where $\boldsymbol{L}$ is a linear operator, and $\mathbf b$ denote the observed data of the same size as $\boldsymbol{ L}\bigxiaokuohao{\bigllbracket{\boldsymbol{ \sigma};U_j}  }$. When $\boldsymbol{L}$ represents the identity operator and $\mathbf b$ denotes $\mathcal A$, the data-fitting term boils down to the objective of \eqref{prob:robust_orth_main}. When  $\boldsymbol{ \Phi}_\delta\bigxiaokuohao{\boldsymbol{L}\bigxiaokuohao{ \bigllbracket{\boldsymbol{ \sigma};U_j}  }-\mathbf b} = \boldsymbol{ \Phi}_\delta(\boldsymbol{\Omega}\circledast\bigxiaokuohao{ \bigllbracket{\boldsymbol{ \sigma};U_j} -\mathcal A  }    )$, where $\boldsymbol{\Omega} \in\T$ is a given $0-1$ tensor with $\boldsymbol{\Omega}_{i_1\cdots i_d}=1$ if $\mathcal A_{i_1\cdots i_d}$ being observed while $\boldsymbol{\Omega}_{i_1\cdots i_d}=0$ if $\mathcal A_{i_1\cdots i_d}$ missing,  it can be used to deal with robust tensor approximation with incomplete data. When $\boldsymbol{L}$ is formed by a set of input data tensors, and each entry of $\mathbf b$ denotes the output score of the corresponding input data tensor, it is the objective of the (robust) tensor regression problem. To minimize $\boldsymbol{ \Phi}_\delta\bigxiaokuohao{\boldsymbol{L}\bigxiaokuohao{ \bigllbracket{\boldsymbol{ \sigma};U_j}  }-\mathbf b}$ over orthonormal constraints, similar to \eqref{prob:robust_orth_main_P}, one can also formulate the problem as
\begin{equation*}  
\setlength\abovedisplayskip{4pt}
\setlength\abovedisplayshortskip{4pt}
\setlength\belowdisplayskip{4pt}
\setlength\belowdisplayshortskip{4pt}
\begin{split}
&\min_{\boldsymbol{ \sigma},U_j,\mathcal T,\mathbf w}~     \boldsymbol{ \Phi}_\delta\bigxiaokuohao{\boldsymbol{L}\bigxiaokuohao{ \mathcal T  }-\mathbf b}= \frac{1}{2}  \normF{ \sqrt \mathbf w  \circledast\bigxiaokuohao{\boldsymbol{L}\bigxiaokuohao{ \mathcal T }-\mathbf b}  }^2 + \frac{\delta^2}{2}\sum_{i =1} \nolimits  \varrho(\mathbf w_{i})   \\
&~~~~    {\rm s.t.}~  \mathcal T = \llbracket \boldsymbol{ \sigma}; U_j\rrbracket,~\mathbf w\geq 0,\\
&~~~~~~~~~  \mathbf u_{j,i}^\top \mathbf u_{j,i} =1, 1\leq j\leq d-t ,  1\leq i \leq R,\\
& ~~~~~~~~~ U_j^\top U_j = I,  d-t+1\leq j\leq d,
\end{split}
\end{equation*}
where $\mathbf w$ is the same size as $\mathbf b$ defined similar to that in \eqref{eq:cost}. The framework of HQ-ADMM then applies as well. 
	
2.	The idea of combing HQ property and ADMM framework can also be extended to solve other Cauchy loss based problems  such as those studied in \cite{mei2018cauchy,kim2020cauchy}. Specifically, for problems of the form 
	\begin{equation*}\label{eq:tmp}
	\min_{\mathbf x} \boldsymbol{ \Phi}_\delta(L\mathbf x-\mathbf b) + R(\mathbf x),
	\end{equation*}
	where $L$ is a matrix, $\mathbf b$ is a vector of proper size,   one can also   convert it to 
		\[
	\min_{\mathbf x,\mathbf w} \bigfnorm{\sqrt{ \mathbf w }\circledast\bigxiaokuohao{ L\mathbf y-\mathbf b  }  }^2 + \sum_{i=1}\nolimits\varrho(\mathbf w_i) +R(\mathbf x),~{\rm s.t.}~ \mathbf x=\mathbf y,
	\]
with $\mathbf w$ defined similar to that in \eqref{eq:cost};	  an algorithm in the spirit of HQ-ADMM can be applied to solve it. 

	3. An alternative way to obtain closed-form solutions in ADMM for solving \eqref{prob:robust_orth_main} is to use a linearization technique. For example, one can apply a linearized ADMM to solve   the original problem \eqref{prob:robust_orth_main} instead of the equivalent form \eqref{prob:robust_orth_main_P}, in which one also replace $\bigllbracket{\boldsymbol{ \sigma}; U_j}$ by $\mathcal T$; however, to solve the $\mathcal T$-subproblem, i.e., $\min_{\mathcal T}\boldsymbol{ \Phi}_\delta(\mathcal T-\mathcal A) + \innerprod{\mathcal Y}{\mathcal T}+\tau/2\bigfnorm{\mathcal T - \bigllbracket{\boldsymbol{ \sigma};U_j}}^2  $,  which does not admit a closed-form solution, one   linearizes      $\boldsymbol{ \Phi}_\delta(\mathcal T-\mathcal A)$ and then imposes a proximal term. The issue is that by doing this, one does not fully explore the structure of the model, which may lead to inefficiency. On the other hand, extra effort has to be paid to find a suitable step-size for this linearized subproblem.
\end{remark}

\section{Convergence of HQ-ADMM} \label{sec:conv}This section establishes the convergence of HQ-ADMM. We note that to ensure the convergence, the only requirement is that $\tau\geq \sqrt{10}$. 
Throughout this section, to simplify the notations, we denote 
\[
\deltaUj{k+1}{k}:= U^{k+1}_j - U^k_j. 
\]
The definitions of $\deltaP{k+1}{k}$, $\deltaW{k+1}{k}$, and $\deltaY{k+1}{k}$ are analogous.
In addition, we define the following proximal augmented Lagrangian function
   $$\tilde L_\tau(    \boldsymbol{ \sigma},U_j,\mathcal T,\mathcal Y,\mathcal W,\mathcal T^\prime ):= L_\tau(\boldsymbol{ \sigma},U_j,\mathcal T,\mathcal Y,\mathcal W  ) + \frac{2 }{  \tau }\|\mathcal T - \mathcal T^\prime\|_F^2,$$
   which is needed to study  the diminishing property of the terms $\bigfnorm{\deltaUj{k+1}{k}}$ and $\bigfnorm{\deltaP{k+1}{k}}$.
 For convenience we also denote
	\begin{equation}\label{eq:tilde_L}
	\tilde L^{k+1,k}_\tau := \tilde L_\tau(  \boldsymbol{ \sigma}^{k+1},U_j^{k+1},\mathcal T^{k+1},\mathcal Y^{k+1},\mathcal W^{k+1},\mathcal T^{k}).
	\end{equation}

 We present the first main result  in the following, showing that the sequence generated by the algorithm is bounded, and every limit point of the sequence generated by HQ-ADMM is a stationary point. The proof is left to Section \ref{sec:proof_sub_conv}.
\begin{theorem}[Subsequential convergence]\label{th:sub_convergence}
	Let  $\{ \boldsymbol{ \sigma}^k, U^k_j,\mathcal T^k,\mathcal Y^k,\mathcal W^k \}$ be   generated by Algorithm \ref{alg:hq_admm} with $\tau\geq \sqrt{10}$ and $\alpha>0$. Then 
	\begin{enumerate}
		\item $\{ \boldsymbol{ \sigma}^k, U^k_j,\mathcal T^k,\mathcal Y^k,\mathcal W^k \}$  is bounded;
		\item 	the sequence $\{ \tilde L_\tau^{k+1,k} \}$ defined in \eqref{eq:tilde_L} is bounded, nonincreasing and convergent;
		\item 	 it holds that
		\begin{equation}\label{eq:thm:1}
		\sum_{k=1}^{\infty} \bigxiaokuohao{\sum^d_{j=1}\bigfnorm{ \deltaUj{k+1}{k}}^2 + \bigfnorm{\deltaP{k+1}{k}}^2}   <+\infty,
		\end{equation}
		and  
		\begin{equation}\label{eq:thm:2}
		\bignorm{ { \deltasigma{k+1}{k}}}  \rightarrow 0, ~\bigfnorm{\deltaW{k+1}{k}} \rightarrow 0, ~\bigfnorm{\llbracket \boldsymbol{ \sigma}^k; U_j^k\rrbracket - \mathcal T^k}    \rightarrow 0.
		\end{equation}  Moreover, every limit point  $\{ \boldsymbol{ \sigma}^*, U^*_j,\mathcal T^*,\mathcal Y^*,\mathcal W^*\}$  satisfies the optimality condition \eqref{eq:kkt}.  In particular, $\{\boldsymbol{ \sigma}^*,U^*_j \}$ is also a stationary point of the original problem \eqref{prob:robust_orth_main}. 
	\end{enumerate}
\end{theorem}

Next, based on the Kurdyka-\L{}ojasiewicz (KL) property \cite{bolte2014proximal} which is widely used for proving the global convergence of nonconvex algorithms, we can show that the whole sequence converges to a single limit point. The proof is left to Sect. \ref{sec:global_conv}.
\begin{theorem}[Global convergence]\label{th:global_conv}
	Under the setting of Theorem \ref{th:sub_convergence}, the whole sequence of $\{U^k_j,\mathcal T^k \}$ converges to a single limit point, i.e.,
	\[
	\lim_{k\rightarrow\infty} U^k_j = U^*_j,~1\leq j\leq d,~~\lim_{k\rightarrow\infty} \mathcal T^k = \mathcal T^*.
	\]
\end{theorem}

\subsection{Proof of Theorem \ref{th:sub_convergence}}\label{sec:proof_sub_conv}
To prove the convergence of a nonconvex ADMM, a key step is to upper bound the size of the successive difference of the dual variables  by that of the primal variables \cite{li2015global,wang2019global,hong2016convergence}. For HQ-ADMM, the weight $\mathcal W^k$ brings barriers in the estimation of the upper bound. Fortunately, this can be overcome by realizing the relations between $\mathcal W^k$, $\mathcal T^k$ and $\mathcal T^{k-1}$ by using Lemma \ref{lem:y_and_p}. The resulting estimate is given as follows.
\begin{lemma}   \label{lem:y_and_p}
 It holds that
	\[
	\| \deltaY{k+1}{k}\|_F \leq  \|  \deltaP{k+1}{k}  \|_F + \| \deltaP{k}{k-1}  \|_F  .
	\]
\end{lemma}
\begin{proof} 
From \eqref{prob:P_subproblem}, we have
	\begin{equation*}  
				 \mathcal W^k\circledast\bigxiaokuohao{  \mathcal T^{k+1}-\mathcal A }   + \mathcal Y^k    -  \tau  \bigxiaokuohao{ \llbracket\boldsymbol{ \sigma}^{k};U^{k+1}_j\rrbracket - \mathcal T^{k+1}    }=0,
	\end{equation*}
	which together with the definition of $\mathcal Y^{k+1}$ yields
	\begin{equation}
	\label{eq:conv:1}
 \mathcal W^k\circledast\bigxiaokuohao{  \mathcal T^{k+1}-\mathcal A }   + 	\mathcal Y^{k+1} = 0.
	\end{equation}
Therefore we have
	\begin{small}
		\begin{eqnarray} 
		\|\deltaY{k+1}{k}  \| &=&  \bigfnorm{ \mathcal W^k\circledast\bigxiaokuohao{  \mathcal T^{k+1}-\mathcal A }  -     \mathcal W^{k-1}\circledast\bigxiaokuohao{  \mathcal T^{k}-\mathcal A }  }  \nonumber\\
		&= &\bigfnorm{ \mathcal W^k\circledast\bigxiaokuohao{  \mathcal T^{k+1}-\mathcal A } - \mathcal W^k\circledast\bigxiaokuohao{  \mathcal T^{k}-\mathcal A }  +  \mathcal W^k\circledast\bigxiaokuohao{  \mathcal T^{k}-\mathcal A }  -   \mathcal W^{k-1}\circledast\bigxiaokuohao{  \mathcal T^{k}-\mathcal A } }\nonumber\\
		&\leq &\bigfnorm{ \mathcal W^k\circledast\bigxiaokuohao{  \mathcal T^{k+1}-\mathcal T^k }  }  + \bigfnorm{ (\mathcal W^k - \mathcal W^{k-1})\circledast\bigxiaokuohao{  \mathcal T^{k}-\mathcal A }    }  \label{eq:conv:2}
		\end{eqnarray}
	\end{small}
\noindent Now  	denote   $E_1:= \bigfnorm{ \mathcal W^k\circledast\bigxiaokuohao{  \mathcal T^{k+1}-\mathcal T^k }   }  $ and $E_2:=\bigfnorm{ (\mathcal W^k - \mathcal W^{k-1})\circledast\bigxiaokuohao{  \mathcal T^{k}-\mathcal A }    } $. We first consider $E_1$. From the definition of $\mathcal W^k$, we easily see that $\mathcal W^k_{i_1\cdots i_d}\leq 1$ for each $i_1,\ldots,i_d$. Therefore,
	\begin{equation}\label{eq:conv:3}
	E_1 \leq \| \deltaP{k+1}{k}   \|.
	\end{equation}
	Next we focus on $E_2$. To simplify notations we denote $a_{i_1\cdots i_d}^k:= \mathcal T^k_{i_1\cdots i_d} - \mathcal A_{i_1\cdots i_d}$ and 
	$$e_{i_1\cdots i_d}:= \delta^2 a_{i_1\cdots i_d}^k \left(   \frac{1}{ \delta^2 + (a_{i_1\cdots i_d}^k)^2  } - \frac{1}{\delta^2 +  (a_{i_1\cdots i_d}^{k-1})^2  } \right) .$$ 
	Then  $E_2$ can be expressed as
	\begin{eqnarray*}
	E_2^2 &=&\sum^{n_1,\ldots,n_d}_{i_1=1,\ldots,i_d=1} \bigxiaokuohao{ \mathcal W_{i_1\cdots i_d}^{k+1} - \mathcal W_{i_1\cdots i_d}^k  }^2\bigxiaokuohao{ \mathcal T_{i_1\cdots i_d} - \mathcal A_{i_1\cdots i_d}  }^2    \\
	&=&  \sum^{n_1,\ldots,n_d}_{i_1=1,\ldots,i_d=1}  \delta^4 (a_{i_1\cdots i_d}^k)^2 \left(   \frac{1}{ \delta^2 + (a_{i_1\cdots i_d}^k)^2  } - \frac{1}{\delta^2 +  (a_{i_1\cdots i_d}^{k-1})^2  } \right)^2 = \sum^{n_1,\ldots,n_d}_{i_1=1,\ldots,i_d=1} e_{i_1\cdots i_d}^2.
	\end{eqnarray*}
It follows from Proposition \ref{prop:E2} that
	\begin{equation*}  
	\label{eq:conv:4}
	|e_{i_1\cdots i_d}| \leq | a_{i_1\cdots i_d}^k-a_{i_1\cdots i_d}^{k-1}  |,  \end{equation*}
	and so
	\begin{equation}\label{eq:conv:5}
	E_2 \leq \| \mathcal T^k-\mathcal A - (\mathcal T^{k-1} - \mathcal A)  \|_F = \|\deltaP{k}{k-1}\|_F.
	\end{equation}
	\eqref{eq:conv:2} combining with \eqref{eq:conv:3} and \eqref{eq:conv:5} yields the desired result.
\end{proof}

With Lemma \ref{lem:y_and_p}, we then establish a sufficiently decreasing inequality with respect to $\{\tilde{L}_\tau^{k+1,k}  \}$ defined in \eqref{eq:tilde_L}.
\begin{lemma}\label{lem:suff_decrease_1}
	Let the parameter $\tau$ satisfy $ \tau \geq \sqrt{10}$. Then there holds
	\begin{equation*}
	\label{eq:lemma:suff_decrease:7}
	\tilde L_\tau^{k,k-1} - \tilde L_\tau^{k+1,k} \geq \frac{\alpha}{2}\sum^d_{j=1}\bigfnorm{\deltaUj{k+1}{k} }^2+ \frac{1}{\tau} \bigfnorm{ \deltaP{k+1}{k} }^2,~\forall k,
	\end{equation*}
	where $\alpha>0$ is defined in \eqref{prob:u_subproblem} and \eqref{prob:U_subproblem}. 
\end{lemma}
\begin{proof}
	We first consider the decrease caused by $U_j$. When $1\leq j \leq d-t$, according to the algorithm, the expression of $L_\tau(\cdot)$, that $\bignorm{\mathbf u^k_{j,i} }=1$ and recalling the definition of $\mathbf u^{k+1}_{j,i}$, $\mathbf v^{k+1}_{j,i}$ and $\tilde{\mathbf v}^{k+1}_{j,i}$, we have
	\begin{eqnarray}
&&	L_\tau(\boldsymbol{ \sigma}^k, U^{k+1}_1,\ldots,U^{k+1}_{j-1},U^k_j,\ldots,U^d_j  ,\mathcal T^k,\mathcal Y^k,\mathcal W^k ) - \nonumber\\
&&~~~~~~~~~~~~~~~ L_\tau(\boldsymbol{ \sigma}, U^{k+1}_1,\ldots, U^{k+1}_j,U^k_{j+1},\ldots,U^k_d,\mathcal T^k,\mathcal Y^k,\mathcal W^k) \nonumber\\
=&& \sum^R_{i=1}\innerprod{ \sigma^k_i \cdot (\mathcal Y^k + \tau\mathcal T^k) {\mathbf u_{1,i}^{k+1}\otimes \cdots\otimes \mathbf u_{j-1,i}^{k+1} \otimes \mathbf u_{j+1,i}^k \otimes\cdots\otimes \mathbf u_{d,i}^k  }   }{\mathbf u^{k+1}_{j,i} - \mathbf u^k_{j,i}} \nonumber\\
=&&\sum^R_{i=1}\innerprod{\sigma^k_i\cdot \mathbf v^{k+1}_{j,i}}{ \mathbf u^{k+1}_{j,i}  - \mathbf u^k_{j,i}   } \nonumber\\
=&& \sum^R_{i=1} \innerprod{ \sigma^k\cdot \mathbf v^{k+1} + \alpha \mathbf u^k_{j,i}   }{\mathbf u^{k+1}_{j,i} - \mathbf u^k_{j,i}   } + \frac{\alpha}{2}\bignorm{ \mathbf u^{k+1}_{j,i} - \mathbf u^k_{j,i}  }^2 \nonumber\\
=&& \sum^R_{i=1} \innerprod{\tilde{\mathbf v}^{k+1}_{j,i}}{ \frac{\tilde{\mathbf v}^{k+1}_{j,i}}{\bignorm{\tilde{\mathbf v}^{k+1}_{j,i}}   } -\mathbf u^k_{j,i}     }+  \frac{\alpha}{2}\bignorm{ \mathbf u^{k+1}_{j,i} - \mathbf u^k_{j,i}  }^2 \nonumber\\
 \geq &&   \frac{\alpha}{2}\sum^R_{i=1}\bignorm{ \mathbf u^{k+1}_{j,i} - \mathbf u^k_{j,i}  }^2  = \frac{\alpha}{2}\bignorm{\deltaUj{k+1}{k}}^2_F ,\label{eq:conv:6}
	\end{eqnarray}
	where the fourth equality follows from the definition of $\mathbf u^{k+1}_{j,i}$ and $\tilde{\mathbf v}^{k+1}_{j,i}$, and the inequality is due to $\bignorm{\mathbf v}\geq \innerprod{\mathbf v}{\mathbf u}$ for any vectors $\mathbf u,\mathbf v$ of the same size with $\bignorm{\mathbf u}=1$. 
	
	The decrease of $U_j$ when $d-t+1\leq j\leq d$ is similar.  From the definition of $V^{k+1}_j$,  It holds that
		\begin{eqnarray}
	&&	L_\tau(\boldsymbol{ \sigma}^k, U^{k+1}_1,\ldots,U^{k+1}_{j-1},U^k_j,\ldots,U^k_d  ,\mathcal T^k,\mathcal Y^k,\mathcal W^k ) - \nonumber\\
	&&~~~~~~~~~~~~~~~ L_\tau(\boldsymbol{ \sigma}^k, U^{k+1}_1,\ldots, U^{k+1}_j,U^k_{j+1},\ldots,U^k_d,\mathcal T^k,\mathcal Y^k,\mathcal W^k) \nonumber\\
	=&& \sum^R_{i=1}\innerprod{ \sigma^k_i \cdot (\mathcal Y^k + \tau\mathcal T^k) {\mathbf u_{1,i}^{k+1}\otimes \cdots\otimes \mathbf u_{j-1,i}^{k+1} \otimes \mathbf u_{j+1,i}^k \otimes\cdots\otimes \mathbf u_{d,i}^k  }   }{\mathbf u^{k+1}_{j,i} - \mathbf u^k_{j,i}} \nonumber\\
	=&&\innerprod{ V^{k+1}_j \cdot {\rm diag}(\boldsymbol{ \sigma}^k)     }{U^{k+1}_j - U^k_j} \nonumber\\
	=&&\innerprod{ V^{k+1}_j \cdot {\rm diag}(\boldsymbol{ \sigma}^k)  + \alpha U^k_j   }{U^{k+1}_j - U^k_j}  + \frac{\alpha}{2}\bigfnorm{ U^{k+1}_j - U^k_j }^2 \nonumber\\
	\geq &&  \frac{\alpha}{2}\bigfnorm{ \deltaUj{k+1}{k}}^2  ,\label{eq:conv:7}
	\end{eqnarray}
	where the inequality follows from the definition of $U^{k+1}_j$ in \eqref{prob:U_subproblem}.

To show the decrease of $\mathcal T$, note that $L_\tau(\cdot)$ is strongly convex with respect to $\mathcal T$, which we can easily deduce that
\begin{equation}
\label{eq:conv:9}
L_\tau(\boldsymbol{ \sigma}^{k},U^{k+1}_j, \mathcal T^k,\mathcal Y^k,\mathcal W^k) - L_\tau(\boldsymbol{ \sigma}^{k}, U^{k+1}_j,\mathcal T^{k+1},\mathcal Y^k,\mathcal W^k) \geq \frac{\tau}{2}\bignorm{\deltaP{k+1}{k}  }^2_F.
\end{equation}
Next, it follows from the definition of $\mathcal Y^{k+1}$ and Lemma \ref{lem:y_and_p} that 
\begin{eqnarray}
&& L_\tau(\boldsymbol{ \sigma}^{k}, U^{k+1}_j,\mathcal T^{k+1},\mathcal Y^k,\mathcal W^k) - L_\tau(\boldsymbol{ \sigma}^{k}, U^{k+1}_j,\mathcal T^{k+1},\mathcal Y^{k+1},\mathcal W^k) \nonumber\\
=&& \innerprod{ \mathcal Y^{k+1} - \mathcal Y^k  }{ \llbracket \boldsymbol{ \sigma}^{k}; U^{k+1}_j \rrbracket - \mathcal T^{k+1}  } \nonumber\\
=&& -\frac{1}{\tau}\bigfnorm{ \deltaY{k+1}{k}   }^2\nonumber\\
\geq &&  -\frac{2}{\tau}\bigxiaokuohao{ \bigfnorm{\deltaP{k+1}{k}}^2 + \bigfnorm{\deltaP{k}{k-1}}^2      }.   \label{eq:conv:10}
 \end{eqnarray}
 Finally, it follows from the definition of $\boldsymbol{ \sigma}^{k+1}$ and $\mathcal W^{k+1}$ that
 \begin{equation}
 \label{eq:conv:11}
 L_\tau(\boldsymbol{ \sigma}^{k}, U^{k+1}_j, \mathcal T^{k+1},\mathcal Y^{k+1},\mathcal W^k)  - L_\tau(\boldsymbol{ \sigma}^{k+1}, U^{k+1}_j, \mathcal T^{k+1},\mathcal Y^{k+1},\mathcal W^{k+1}) \geq 0.
 \end{equation}
 As a result, summing up \eqref{eq:conv:6}--\eqref{eq:conv:11} yields
\begin{eqnarray}
&& L_\tau(\boldsymbol{ \sigma}^k, U^k_j,\mathcal T^k,\mathcal Y^k,\mathcal W^k) - L_\tau(\boldsymbol{ \sigma}^{k+1}, U^{k+1}_j,\mathcal T^{k+1},\mathcal Y^{k+1},\mathcal W^{k+1}) \nonumber\\
\geq && \frac{\alpha}{2}\sum^d_{j=1}\bigfnorm{ \deltaUj{k+1}{k}}^2 + \bigxiaokuohao{ \frac{\tau}{2} - \frac{2}{\tau}   }\bigfnorm{ \deltaP{k+1}{k}}^2 - \frac{2}{\tau}\bigfnorm{\deltaP{k}{k-1}}^2 \nonumber\\
\geq && \frac{\alpha}{2}\sum^d_{j=1}\bigfnorm{ \deltaUj{k+1}{k}}^2 + \bigxiaokuohao{   \frac{2}{\tau}  + \frac{1}{\tau} }\bigfnorm{ \deltaP{k+1}{k}}^2 - \frac{2}{\tau}\bigfnorm{\deltaP{k}{k-1}}^2,\label{eq:conv:12}
\end{eqnarray}
where the last inequality follows from the range of $\tau$. Rearranging the terms of \eqref{eq:conv:12} gives the desired results. This completes the proof.
\end{proof}

We then show that $\tilde L_\tau^{k,k-1}$ defined in Lemma \ref{lem:suff_decrease_1} is lower bounded and the sequence $\{ \boldsymbol{ \sigma}^k,U^k_i,\mathcal T^k,\mathcal Y^k,\mathcal W^k \}$ is bounded as well. 
\begin{theorem}
	\label{prop:boundedness}
	Under the setting of Lemma \ref{lem:suff_decrease_1},	$\{\tilde L_\tau^{k,k-1}\}$   is bounded. The sequence $\{\boldsymbol{ \sigma}^k, U^k_j,\mathcal T^k,\mathcal Y^k,\mathcal W^k \}$   generated by Algorithm \ref{alg:hq_admm} is bounded as well.
\end{theorem}
\begin{proof}
Denote 
$Q^{k}(\mathcal T) := \frac{1}{2}\bigfnorm{ \sqrt{\mathcal W^k}\circledast\bigxiaokuohao{ \mathcal T-\mathcal A  }  }^2 $; thus we have $\nabla Q^k(\mathcal T) = \mathcal W^k\circledast\bigxiaokuohao{\mathcal T-\mathcal A}$, and it then follows from the quadraticity of $Q^k(\cdot)$ and $\mathcal Y^k = -\mathcal W^{k-1}\circledast\bigxiaokuohao{\mathcal T^k-\mathcal A}$ from \eqref{eq:conv:1} that 
\begin{eqnarray}\label{eq:lemma:bounded:1}
&&Q^{k-1}(\mathcal T^{k}) - Q^{k-1}(\llbracket \boldsymbol{ \sigma^k}; U^k_j \rrbracket)- \innerprod{ \mathcal Y^k  }{ \llbracket \boldsymbol{ \sigma^k}; U^k_j \rrbracket - \mathcal T^k   } \nonumber\\
=&& \innerprod{ \mathcal W^{k-1}\circledast\bigxiaokuohao{  \llbracket \boldsymbol{ \sigma^k}; U^k_j \rrbracket  - \mathcal A  }  }{\mathcal T^{k} - \llbracket \boldsymbol{ \sigma^k}; U^k_j \rrbracket }  \nonumber\\
&&~~~~~~~~~~~~+\frac{1}{2}\bigfnorm{ \sqrt{\mathcal W^{k-1}}\circledast\bigxiaokuohao{  \llbracket \boldsymbol{ \sigma^k}; U^k_j \rrbracket - \mathcal T^k   }   }^2- \innerprod{ \mathcal Y^k  }{ \llbracket \boldsymbol{ \sigma^k}; U^k_j \rrbracket - \mathcal T^k   }\nonumber\\
=&& \frac{1}{2}\bigfnorm{ \sqrt{\mathcal W^{k-1}}\circledast\bigxiaokuohao{  \llbracket \boldsymbol{ \sigma^k}; U^k_j \rrbracket - \mathcal T^k   }   }^2 \nonumber\\
&&~~~~~~~~~~~~+ \innerprod{ \mathcal W^{k-1}\circledast\bigxiaokuohao{ \llbracket \boldsymbol{ \sigma^k}; U^k_j \rrbracket  - \mathcal A   } - \mathcal W^{k-1}\circledast\bigxiaokuohao{ \mathcal T^k-\mathcal A   }   }{  \mathcal T^k - \llbracket \boldsymbol{ \sigma^k}; U^k_j \rrbracket     }\nonumber\\
=&& - \frac{1}{2}\bigfnorm{ \sqrt{\mathcal W^{k-1}}\circledast\bigxiaokuohao{  \llbracket \boldsymbol{ \sigma^k}; U^k_j \rrbracket - \mathcal T^k   }   }^2 \geq -\frac{1}{2}\bigfnorm{   \llbracket \boldsymbol{ \sigma^k}; U^k_j \rrbracket - \mathcal T^k  }^2,
\end{eqnarray}
where   the last inequality uses the fact that $0<\mathcal W^{k-1}_{i_1\cdots i_d} \leq 1$. 

Based on \eqref{eq:lemma:bounded:1}, it follows  from  the proof of Lemma \ref{lem:suff_decrease_1} that for any $k\geq 2$,  
\begin{eqnarray}
&& \tilde L_\tau^{k-1,k-2} = \tilde L_\tau(\boldsymbol{ \sigma}^{k-1},U^{k-1}_j,\mathcal T^{k-1},\mathcal Y^{k-1},\mathcal W^{k-1},\mathcal T^{k-2} )\geq 
\tilde L_\tau(\boldsymbol{ \sigma}^{k},U^{k}_j,\mathcal T^{k},\mathcal Y^{k},\mathcal W^{k-1},\mathcal T^{k-1}) \nonumber\\
&=& Q^{k-1}(\mathcal T^k) +  \frac{\delta^2}{2}\sum^{n_1,\ldots,n_d}_{i_1=1,\ldots,i_d=1} \varrho( \mathcal W^{k-1}_{i_1\cdots i_d} )   - \innerprod{\mathcal Y^k}{ \llbracket \boldsymbol{ \sigma^k}; U^k_j \rrbracket  - \mathcal T^k    }  + \frac{\tau}{2}\bigfnorm{ \llbracket \boldsymbol{ \sigma^k}; U^k_j \rrbracket  - \mathcal T^k  }^2 + \frac{2}{\tau}\bigfnorm{\deltaP{k}{k-1}}^2 \nonumber\\
&\geq & Q^{k-1}( \llbracket \boldsymbol{ \sigma^k}; U^k_j \rrbracket )  + \frac{\tau -1}{2}\bigfnorm{   \llbracket \boldsymbol{ \sigma^k}; U^k_j \rrbracket - \mathcal T^k  }^2  +  \frac{\delta^2}{2}\sum^{n_1,\ldots,n_d}_{i_1=1,\ldots,i_d=1} \varrho( \mathcal W^{k-1}_{i_1\cdots i_d} ) + \frac{2}{\tau}\bigfnorm{\deltaP{k}{k-1}}^2 \nonumber\\
&>& -\infty, \label{eq:lemma:bounded:2}
\end{eqnarray}
where the first inequality follows from \eqref{eq:lemma:bounded:1} and the last one is due to the range of $\tau$ and $\varrho(\cdot)\geq 0$.  Thus $\{ \tilde L_\tau^{k,k-1}  \}$ is a lower bounded sequence. This  together with Lemma \ref{lem:suff_decrease_1} shows that $\{ \tilde L_\tau^{k,k-1}  \}$  is bounded.  

We then show the boundedness of $\{\boldsymbol{ \sigma}^k,U^k_j,\mathcal T^k,\mathcal Y^k,\mathcal W^k  \}$. The boundedness of $\{U^k_j \}$ and $\{\mathcal W^k \}$ is obvious. Next, denote $g(\boldsymbol{ \sigma}^k)$ as the formulation in line 3 of \eqref{eq:lemma:bounded:2} with respect to $\boldsymbol{ \sigma}^k$.  Since by Proposition \ref{prop:orth}, namely, the orthonormality of $\bigotimes_{j=1}^d\mathbf u^k_{j,i}$,
$$\bigfnorm{\llbracket \boldsymbol{ \sigma}^k; U^k_j\rrbracket  -\mathcal T^k} = \bignorm{ \boldsymbol{ \sigma}^k }^2 - 2\innerprod{ \llbracket \boldsymbol{ \sigma}^k;U^k_j \rrbracket}{\mathcal T^k} + \bigfnorm{\mathcal T^k}^2,$$ 
while $Q^{k-1}(\llbracket \boldsymbol{ \sigma}^k; U^k_j\rrbracket)$ is convex with respect to $\boldsymbol{ \sigma}^k$, we see that $g(\boldsymbol{ \sigma}^k)$ is strongly convex with respect to $\boldsymbol{ \sigma}^k$. This together with the boundedness of $\{ \tilde L_\tau^{k,k-1 } \}$ and \eqref{eq:lemma:bounded:2} gives the boundedness of $\{\boldsymbol{ \sigma}^k \}$. Quite similarly we have that  $\{\mathcal T^k \}$ is bounded. Finally, the boundedness of $\{\mathcal Y^k \}$ follows from the expression of the $\mathcal T$-subproblem \eqref{prob:P_subproblem}.  As a result, the sequence $\{\boldsymbol{ \sigma}^k,U^k_j,\mathcal T^k,\mathcal Y^k,\mathcal W^k  \}$ is bounded. This completes the proof.
\end{proof}

\begin{proof}[Proof of Theorem \ref{th:sub_convergence}]
 Lemma \ref{lem:suff_decrease_1} in connection with Theorem \ref{prop:boundedness} yields points $1$, $2$, and \eqref{eq:thm:1}; \eqref{eq:thm:1}   together with Lemma \ref{lem:y_and_p} and the definition of $\mathcal Y^{k+1}$, $\boldsymbol{ \sigma}^{k+1}$ and $\mathcal W^{k+1}$ gives \eqref{eq:thm:2}. On the other hand, since the sequence is bounded, limit points exist. Assume that $\{ \boldsymbol{ \sigma}^*, U^*_j,\mathcal T^*,\mathcal Y^*,\mathcal W^*\}$ is a limit point with
	\[
	\lim_{l\rightarrow \infty}  \{ \boldsymbol{ \sigma}^{k_l}, U^{k_l}_j,\mathcal T^{k_l},\mathcal Y^{k_l},\mathcal W^{k_l}\}= \{ \boldsymbol{ \sigma}^*, U^*_j,\mathcal T^*,\mathcal Y^*,\mathcal W^*\}.
	\]
	\eqref{eq:thm:1},  \eqref{eq:thm:2} then implies that
	\[
	\lim_{l\rightarrow \infty} \{ \boldsymbol{ \sigma}^{k_l+1}, U^{k_l+1}_j,\mathcal T^{k_l+1},\mathcal Y^{k_l+1},\mathcal W^{k_l+1}\} = \{ \boldsymbol{ \sigma}^*, U^*_j,\mathcal T^*,\mathcal Y^*,\mathcal W^*\}.
	\]
	Therefore, taking the limit into $l$ with respect to the $\mathbf u_{j,i}$-subproblem  \eqref{prob:u_subproblem} yields
	\begin{equation}\label{eq:sub_conv:1}
	\mathbf v^*_{j,i}\sigma^*_i + \alpha\mathbf u^*_{j,i} = \bignorm{\tilde{\mathbf v}^*_{j,i}}\mathbf u^*_{j,i},~1\leq j\leq d-t,~1\leq i\leq R.
	\end{equation}
	Multiplying both sides by $\mathbf u^*_{j,i}$ gilves
	\begin{equation}\label{eq:sub_conv:2}
	\bignorm{\tilde{\mathbf v}^*_{j,i}  } = \alpha + \sigma^*_i \innerprod{\mathbf v^*_{j,i}  }{\mathbf u^*_{j,i}} = \alpha + \sigma^*_i\innerprod{\mathcal Y^* + \tau\mathcal T^*}{\bigotimes_{j=1}^d\nolimits\mathbf u^*_{j,i}} = \alpha + \tau(\sigma^*_i)^2,
	\end{equation}
	where the second equality follows from the definition of $\mathbf v_{j,i}$ and the last one is given by passing the limit into the expression of $\sigma^{k_l+1}_i$ \eqref{prob:sigma_subproblem}. Thus \eqref{eq:sub_conv:1} together with \eqref{eq:sub_conv:2} gives 
	\begin{equation}
\label{eq:sub_conv:3}
(\mathcal Y^* + \tau\mathcal T^*)\bigotimes_{l\neq j}^d\nolimits\mathbf u_{l,i}^* =    \sigma_i^*\tau\mathbf u_{j,i}^*,   
		\end{equation}
i.e.,	the first equation of the stationary point system \eqref{eq:kkt}. 
	
	Taking the limit into $l$ with respect to the $U_j$-subproblem \eqref{prob:U_subproblem} and noticing the expression \eqref{eq:U_subproblem_1}, we get
	 \begin{equation*}
	 V^*_j{\rm diag}(\boldsymbol{ \sigma}^*) + \alpha U^*_j = U^*_j H^*_j,
	 \end{equation*}
	 where $H^*_j$ is a symmetric matrix. Writing it columnwisely, we obtain
	 \[
	 \sigma^*_i  \bigxiaokuohao{\mathcal Y^*+\tau\mathcal T^*  }\bigotimes_{l\neq j}^d\nolimits\mathbf u^*_{l,i} = \sum^R_{i=1}\nolimits( H^*_j )_{i,r}\mathbf u^*_{j,r} - \alpha \mathbf u^*_{j,i},~d-t+1\leq j\leq d,~1\leq i\leq R.
	 \]
	 Denoting $\Lambda^*_j:= H^*_j - \alpha I$, the above is exactly the third equality of \eqref{eq:kkt}. On the other hand, passing the limit into the expression of $\mathcal T^k$ \eqref{prob:P_subproblem} and $\mathcal W^k$ \eqref{prob:W_subproblem} respectively gives the $\mathcal T^*$- and $\mathcal W^*$- formulas in \eqref{eq:kkt}. Finally, the first expression of \eqref{eq:thm:2} yields $\mathcal T^* = \llbracket\boldsymbol{ \sigma}^*; U^*_j\rrbracket$. Taking the above pieces together, we have that $\{ \boldsymbol{ \sigma}^*, U^*_j,\mathcal T^*,\mathcal Y^*,\mathcal W^*\}$ satisfies the stationary point system \eqref{eq:kkt}. 
	 
	 Next, we show that $\{\boldsymbol{ \sigma}^*,U^*_j  \}$ is also a stationary point of problem \eqref{prob:robust_orth_main}. We define its Lagrangian function as $L_{\boldsymbol{ \Phi}} := \boldsymbol{ \Phi}_\delta(\boldsymbol{ \sigma}, U_j) -   \sum_{j,i=1}^{d-t,R}\nolimits \eta_{j.i}\bigxiaokuohao{ \mathbf u_{j,i}^\top \mathbf u_{j,i} -1} - \sum^{d}_{j=d-t+1}\nolimits\innerprod{\Lambda_j}{ U_j^\top U_j - I}$, similar to that in \eqref{eq:aug_LL}. Taking derivative yields
	 \begin{equation}\label{eq:sub_conv:4}
\footnotesize
\left\{  
\begin{array}{lr}
\partial_{\mathbf u_{j,i}} \boldsymbol{ \Phi}_\delta(\boldsymbol{ \sigma};U_j) = \eta_{j,i}\mathbf u_{j,i} \Leftrightarrow \mathcal W\circledast\bigxiaokuohao{\bigllbracket{\boldsymbol{ \sigma},U_j}-\mathcal A  } \cdot\sigma_i\bigotimes_{l\neq j}\mathbf u_{j,i}= \eta_{j,i}\mathbf u_{j,i} ,&1\leq j\leq d-t,1\leq i\leq R,\\
\partial_{\mathbf u_{j,i}}\boldsymbol{ \Phi}_\delta(\boldsymbol{ \sigma},U_j) = \sum^R_{r=1}(\Lambda_j)_{i,r}\mathbf u_{j,r}  \Leftrightarrow \mathcal W\circledast\bigxiaokuohao{\bigllbracket{\boldsymbol{ \sigma},U_j}-\mathcal A  } \cdot\sigma_i\bigotimes_{l\neq j}\mathbf u_{j,i}=\sum^R_{r=1}(\Lambda_j)_{i,r}\mathbf u_{j,r}, ,&d-t+1\leq j\leq d,1\leq i\leq R,\\
\partial_{\boldsymbol{ \sigma}}\boldsymbol{ \Phi}_\delta(\boldsymbol{ \sigma},U_j) =0\Leftrightarrow \innerprod{\mathcal W\circledast\bigxiaokuohao{ \bigllbracket{\boldsymbol{ \sigma}; U_j}  -\mathcal A}  }{\bigotimes_{j=1}^d\nolimits \mathbf u_{j,i}} = 0,~&1\leq i\leq R,
\end{array}  
\right.
	 \end{equation}
where $\mathcal W_{i_1\cdots i_d} = \delta^2\bigxiaokuohao{1+ \bigxiaokuohao{\bigllbracket{\boldsymbol{ \sigma}; U_j}_{i_1\cdots i_d}-\mathcal A_{i_1\cdots i_d}   }^2/\delta^2   }^{-1}$; multiplying $\mathbf u_{j,i}$ in both sides of the first equality above, and noticing the last equality, we get $\eta_{j,i}=0$.

Since $\mathcal T^* = \bigllbracket{\boldsymbol{ \sigma}^*;U^*_j}$, the $\mathcal T$-subproblem \eqref{prob:P_subproblem} also gives $\mathcal Y^* = \mathcal W^*\circledast\bigxiaokuohao{\bigllbracket{\boldsymbol{ \sigma}^*;U^*_j} -\mathcal A}$. This together with \eqref{eq:sub_conv:3} and that $\mathcal T^*\bigotimes_{l\neq j}^d\mathbf u^*_{j,i} = \bigllbracket{\boldsymbol{ \sigma}^*;U^*_j}\bigotimes_{l\neq j}^d\mathbf u^*_{j,i}=\sigma^*_i\mathbf u^*_{j,i}$  gives  $\mathcal W^*\circledast\bigxiaokuohao{\bigllbracket{\boldsymbol{ \sigma}^*,U_j^*}-\mathcal A  } \bigotimes_{l\neq j}\mathbf u_{j,i}^*=0$, i.e., the first equality of \eqref{eq:sub_conv:4} by noticing $\eta_{j,i}=0$. In a similar vein, we get that 
\[
\sigma^*_i\mathcal W^*\circledast\bigxiaokuohao{\bigllbracket{\boldsymbol{ \sigma}^*; U^*_j}-\mathcal A} = \sum^R_{i=1}\nolimits( H^*_j )_{i,r}\mathbf u^*_{j,r} - (\alpha+ \tau\sigma^*_i) \mathbf u^*_{j,i} .
\]
Taking $\Lambda_j := H_j^*- (\alpha+\tau\sigma^*_i)I$ gives the second relation of \eqref{eq:sub_conv:4}. The last equality follows directly from $\mathcal W^*\circledast\bigxiaokuohao{\bigllbracket{\boldsymbol{ \sigma}^*,U_j^*}-\mathcal A  } \bigotimes_{l\neq j}\mathbf u_{j,i}^*=0$.
	 The proof has been completed.
\end{proof}

\subsection{Proof of Theorem \ref{th:global_conv}}\label{sec:global_conv}
To prove Theorem \ref{th:global_conv},
  we first recall some definitions from nonsmooth analysis. Denote ${\rm dom}f:=\{\mathbf x\in\mathbb R^n\mid f(\mathbf x)<+\infty \}$. 
\begin{definition}[c.f. \cite{attouch2013convergence}]
	\label{def:subdifferential}
	For $\mathbf x\in {\rm dom}f$, the Fr\'{e}chet subdifferential, denoted as $\hat \partial f(\mathbf x)$, is the  set of vectors $\mathbf z\in\mathbb R^n$ satisfying 
	\begin{equation}\label{eq:f_subdiff}			      \setlength\abovedisplayskip{2pt}
	\setlength\abovedisplayshortskip{2pt}
	\setlength\belowdisplayskip{2pt}
	\setlength\belowdisplayshortskip{2pt}
	\liminf _{\mathbf y \neq \mathbf  x \atop\mathbf  y \rightarrow\mathbf  x} \frac{f(\mathbf  y)-f(\mathbf  x)-\langle \mathbf z, \mathbf  y-\mathbf  x\rangle}{\|\mathbf  x-\mathbf  y\|}\geq 0.
	\end{equation}
	The subdifferential of $f$ at $\mathbf x\in{\rm dom}f$, written $\partial f$,  is defined as
	\[ \setlength\abovedisplayskip{3pt}
	\setlength\abovedisplayshortskip{3pt}
	\setlength\belowdisplayskip{3pt}
	\setlength\belowdisplayshortskip{3pt}
	\partial f(\mathbf x):=\left\{\mathbf z \in \mathbb{R}^{n}: \exists \mathbf x^{k} \rightarrow \mathbf x, f\left(\mathbf x^{k}\right) \rightarrow f(\mathbf x), \mathbf z^{k} \in \hat{\partial} f\left(\mathbf x^{k}\right) \rightarrow \mathbf z\right\}.
	\]
\end{definition}
It is known that $\hat\partial f(\mathbf x)\subset \partial f(\mathbf x)$ for each $\mathbf x\in\mathbb R^n$ \cite{bolte2014proximal}.  
An extended-real-valued function is a function $f:\mathbb R^n\rightarrow [-\infty,\infty]$, which is proper if $f(\mathbf x)>-\infty$ for all $\mathbf x$ and $f(\mathbf x)<\infty$ for at least one $\mathbf x$. It is called closed if it is lower semi-continuous (l.s.c. for short).
The global convergence relies on the  the Kurdyka-\L{}ojasiewicz (KL) property given as follows:
\begin{definition}[KL property and KL function, c.f.  \cite{bolte2014proximal,attouch2013convergence}]\label{def:kl} A proper function $f$ is said to have the KL property at $\overline{\mathbf x}\in {\rm dom}\partial f :=\{\mathbf x\in\mathbb R^n\mid \partial f(\mathbf x)\neq\emptyset  \}$, if there exist $\bar\epsilon\in(0,\infty]$, a neighborhood $\mathcal N$ of $\overline{\mathbf x}$, and a continuous and concave function   $\psi: [0,\bar\epsilon) \rightarrow \mathbb R_+$ which is continuously differentiable on $(0,\bar\epsilon)$ with positive derivatives and $\psi(0)=0$, such that for all $\mathbf x\in \mathcal N$ satisfying $f( \overline{\mathbf x}) <f({\mathbf x}) < f(\overline{\mathbf x}) + \bar\epsilon $, it holds that
	\begin{equation*}\label{eq:kl_property}			      \setlength\abovedisplayskip{3pt}
	\setlength\abovedisplayshortskip{3pt}
	\setlength\belowdisplayskip{3pt}
	\setlength\belowdisplayshortskip{3pt}
	\psi^\prime(f(\mathbf x) - f(\overline{\mathbf x})  ){\rm dist}(0,\partial f(\mathbf x)) \geq 1,
	\end{equation*}
	where ${\rm dist}(0,\partial f(\mathbf x)) $ means the distance from the original point to the set  $\partial f(\mathbf x)$.	If a proper and l.s.c. function $f$ satisfies the KL property at each point of ${\rm dom}\partial f$, then $f$ is called a KL function. 
\end{definition}

We then simplify $\tilde L_\tau(\cdot)$ by eliminating the variables $\mathcal W$ and $\boldsymbol{ \sigma}$. First, from the definition of $\mathcal W^{k+1}$ and Lemma \ref{lem:hq}, we have that
\[
\bigfnorm{\sqrt \mathcal W^{k+1} \circledast\bigxiaokuohao{ \mathcal T^{k+1} -\mathcal A  } }^2 + \delta^2\sum^{n_1,\ldots,n_d}_{i_1=1,\ldots,i_d=1}\varrho(\mathcal W^{k+1}_{i_1\cdots i_d}) = \boldsymbol{ \Phi}_\delta(\mathcal T^{k+1}-\mathcal A ),
\]
where $\boldsymbol{ \Phi}_\delta(\cdot )$ is defined in \eqref{prob:robust_orth_main}. This eliminate the $\mathcal W$ from $\tilde L_\tau(\cdot)$. On the other hand, it follows from the definition of $\boldsymbol{ \sigma}^{k+1}$ \eqref{prob:sigma_subproblem} that
\begin{eqnarray*}
&&-\innerprod{\mathcal Y^{k+1}}{ \llbracket\boldsymbol{ \sigma}^{k+1};U^{k+1}_j \rrbracket - \mathcal T^{k+1}  } + \frac{\tau}{2}\bigfnorm{ \llbracket\boldsymbol{ \sigma}^{k+1};U^{k+1}_j \rrbracket - \mathcal T^{k+1}     }^2 \\
=&& \innerprod{\mathcal Y^{k+1}}{\mathcal T^{k+1}} + \frac{\tau}{2}\bigfnorm{\mathcal T^{k+1}}^2 - \frac{1}{2\tau}\sum^R_{i=1}\bigxiaokuohao{ \bigxiaokuohao{ \mathcal Y^{k+1}+ \tau\mathcal T^{k+1}}\bigotimes_{j=1}^d\mathbf u_{j,i}^{k+1}   }^2.
\end{eqnarray*}
Thus $\sigma$ is also eliminated. In what follows, whenever necessary,   $\sigma^k_i $  still represents the expression $  (\mathcal Y^k+\tau\mathcal T^k)\bigotimes_{j=1}^d\mathbf u^k_{j,i}/\tau$, but we only treat it as a representation instead of a variable.

Then $\tilde L_\tau(\boldsymbol{ \sigma}^{k+1}, U^{k+1}_j, \mathcal T^{k+1},\mathcal Y^{k+1},\mathcal W^{k+1} \mathcal T^k)$ can be equivalently written as
\begin{eqnarray*}
\label{eq:tilde_L_equiv}
&&\tilde L_\tau(U^{k+1}_j,\mathcal T^{k+1},\mathcal Y^{k+1},\mathcal T^k ) \\
=&& \frac{1}{2}\boldsymbol{ \Phi}_\delta(\mathcal T^{k+1} - \mathcal A) + \innerprod{\mathcal Y^{k+1}}{\mathcal T^{k+1}} + \frac{\tau}{2}\bigfnorm{\mathcal T^{k+1}}^2 - \frac{1}{2\tau}\sum^R_{i=1}\bigxiaokuohao{ \bigxiaokuohao{ \mathcal Y^{k+1}+ \tau\mathcal T^{k+1}}\bigotimes_{j=1}^d\mathbf u_{j,i}^{k+1}   }^2 + \frac{2}{\tau}\bigfnorm{  \deltaP{k+1}{k} }^2.\nonumber
\end{eqnarray*}
In addition, we denote
\begin{equation*}
\label{eq:tilde_L_alpha}
\tilde L_{\tau,\alpha}(U_j,\mathcal T,\mathcal Y,\mathcal T^\prime) := \tilde L_\tau(U_j,\mathcal T,\mathcal Y,\mathcal T^\prime) - \frac{\alpha}{2}\sum^d_{j=1}\bigfnorm{ U_j}^2 + \sum^{d-t,R}_{j=1,i=1} \iota_{ \stmanifold{n_j}{1} }(\mathbf u_{j,i}) + \sum^{d}_{j=d-t+1}\iota_{ \stmanifold{n_j}{R}}(U_j).
\end{equation*}
We can see that under the constraints of the optimization problem \eqref{prob:robust_orth_main}, $\tilde L_{\tau,\alpha}(\cdot) = \tilde L_\tau(\cdot) + c$ where $c$ is a constant. This together with Theorem \ref{th:sub_convergence} shows that $\{\tilde L_{\tau,\alpha}(U_j^{k+1},\mathcal T^{k+1},\mathcal Y^{k+1},\mathcal T^k),  \}$ is also a bounded and nonincreasing sequence. In addition, 
we have that $\tilde L_{\tau,\alpha}(\cdot)$ is a KL function.
\begin{proposition}
	\label{prop:KL_L_tau_alpha}
	$\tilde L_{\tau,\alpha}(U_j,\mathcal T,\mathcal Y,\mathcal T^\prime) $ defined above is a proper, l.s.c., and KL function.
\end{proposition}
\begin{proof}
It is clear that $\tilde L_{\tau,\alpha}(\cdot)$ is  proper and l.s.c.. Next,	
since the constrained sets in \eqref{prob:robust_orth_main} are all Stiefel manifolds,   items 2 and 6 of \cite[Example 2]{bolte2014proximal} tell 	 us that   they are   semi-algebraic sets, and their indicator functions are semi-algebraic functions. Therefore, the indicator functions are KL functions \cite[Theorem 3]{bolte2014proximal}. On the other hand,   the remaining part of $\tilde L_{\tau,\alpha}$ (besides the indicator functions) is an analytic function and hence it is KL \cite{bolte2014proximal}. As a result, 	$\tilde L_{\tau,\alpha}(U_j,\mathcal T,\mathcal Y,\mathcal T^\prime) $ is a KL function.
\end{proof}

In the sequel, we mainly rely on $\tilde L_{\tau,\alpha}(\cdot)$ to prove the global convergence. 
For convenience, we denote
\[
\tilde L_{\tau,\alpha}^{k+1,k} := \tilde L_{\tau,\alpha}(U_j^{k+1},\mathcal T^{k+1},\mathcal Y^{k+1},\mathcal T^k), ~{\rm and}~ \partial \tilde L_{\tau,\alpha}^{k+1,k}:= \partial \tilde L_{\tau,\alpha}(U_j^{k+1},\mathcal T^{k+1},\mathcal Y^{k+1},\mathcal T^k);
\]
denote $\deltaUjP{k+1}{k}:= (U_j^{k+1} , \mathcal T^{k+1}) - (U_j^{k},\mathcal T^k)$, and  
\begin{equation*}\label{eq:global_conv:20}
\bigfnorm{\deltaUjP{k+1}{k}  } := \sqrt{\sum^d_{j=1}\nolimits\bigfnorm{\deltaUj{k+1}{k}}^2 + \bigfnorm{\deltaP{k+1}{k}}^2    }.
\end{equation*}

\begin{lemma}
	\label{lem:bounded_partial_L} There exists a large enough constant $c_0>0$, such that
	\begin{equation}\label{eq:global_conv:12_dist_2_partial}
	{\rm dist}(\boldsymbol{0}, \partial \tildeLtaualpha{k+1}{k} ) \leq c_0\bigxiaokuohao{ \bigfnorm{ \deltaUjP{k+1}{k}} + \bigfnorm{\deltaUjP{k}{k-1}}     }.
	\end{equation}
\end{lemma}
\begin{proof} We first consider $\partial_{\mathbf u_{j,i}} \tildeLtaualpha{k+1}{k} $, $1\leq j\leq d-t$, $1\leq i\leq R$, and $\partial_{U_j} \tildeLtaualpha{k+1}{k}  $, $d-t+1\leq j\leq d$, respectively. In what follows, we denote
\[
\overline{\mathbf v}^{k+1}_{j,i}:= \sigma^{k+1}_i \bigxiaokuohao{\mathcal Y^{k+1}+\tau\mathcal T^{k+1}}\bigotimes_{l\neq j}^d\nolimits\mathbf u^{k+1}_{l,i} + \alpha\mathbf u^{k+1}_{j,i}, ~{\rm and}~ \overline V^{k+1}_j := [\bar{\mathbf v}^{k+1}_{j,1},\ldots,\bar{\mathbf v}^{k+1}_{j,R}  ]. 
\]
We also recall $\mathbf v_{j,i}^{k+1}:= (\mathcal Y^k+ \tau\mathcal T^k){\mathbf u_{1,i}^{k+1}\otimes \cdots\otimes \mathbf u_{j-1,i}^{k+1} \otimes \mathbf u_{j+1,i}^k \otimes\cdots\otimes \mathbf u_{d,i}^k  }$ and $\tilde{\mathbf v}_{j,i}^{k+1} = \sigma^k_i \mathbf v^{k+1}_{j,i} + \alpha \mathbf u^k_{j,i}$ for later use. In addition, denote $\tilde V^{k+1}_j := [\tilde{\mathbf v}^{k+1}_{j,1},\ldots,\tilde{\mathbf v}^{k+1}_{j,R}]$.

For $1\leq j\leq d-t$,  one has
\begin{eqnarray}\label{eq:global_conv:0}
\partial_{\mathbf u_{j,i}}\tildeLtaualpha{k+1}{k}  &=& -\sigma^{k+1}_i \bigxiaokuohao{\mathcal Y^{k+1}+\tau\mathcal T^{k+1}}\bigotimes_{l\neq j}^d\nolimits\mathbf u^{k+1}_{l,i}- \alpha\mathbf u^{k+1}_{j,i} + \partial \iota_{\stmanifold{n_j}{1} }(\mathbf u^{k+1}_{j,i})\nonumber\\
&=& - \overline{\mathbf v}^{k+1}_{j,i} + \partial \iota_{ \stmanifold{n_j}{1}}(\mathbf u^{k+1}_{j,i}).
\end{eqnarray}
 we then wish to show that
\begin{equation}
\label{eq:global_conv:1}
\tilde{\mathbf v}^{k+1}_{j,i} \in \hat \partial \iota_{\stmanifold{n_j}{1} }(\mathbf u^{k+1}_{j,i}) \subset \partial \iota_{\stmanifold{n_j}{1} }(\mathbf u^{k+1}_{j,i}).
\end{equation}
The proof is similar to that of \cite[Lemma 6.1]{yang2019epsilon}. First, from the definition of $\iota_{\stmanifold{n_j}{1}}(\cdot) $ and $\hat \partial \iota_{\stmanifold{n_j}{1}}(\cdot)$ in \eqref{eq:f_subdiff}, it is not hard to see that if $\mathbf y \not\in \stmanifold{n_j}{1}$, then  \eqref{eq:f_subdiff} clearly holds when $\mathbf z = \tilde{\mathbf v}^{k+1}_{j,i}$; otherwise if  $\mathbf y \in \stmanifold{n_j}{1}$, i.e., $\|\mathbf y\|=1$, then from the definition of $\mathbf u^{k+1}_{j,i}$, we see that
	\[		
\mathbf u^{k+1}_{j,i} =   \arg\max_{ \|\mathbf y\|=1 } \innerprod{\mathbf y}{ \tilde{\mathbf v}^{k+1}_{j,i}  }\Leftrightarrow \langle \tilde{\mathbf v}^{k+1}_{j,i},\mathbf u^{k+1}_{j,i}-\mathbf y\rangle \geq 0,~\forall \|\mathbf y\|=1,
\]
which together with $\iota_{\stmanifold{n_j}{1}}(\mathbf y) = 0$ and $\iota_{\stmanifold{n_j}{1}}(\mathbf u^{k+1}_{j,i})=0$ gives 	
\[			      \setlength\abovedisplayskip{3pt}
\setlength\abovedisplayshortskip{3pt}
\setlength\belowdisplayskip{3pt}
\setlength\belowdisplayshortskip{3pt}
\liminf _{\mathbf y \neq \mathbf   u^{k+1}_{j,i}, \mathbf  y \rightarrow\mathbf   u^{k+1}_{j,i} } \frac{     \iota_{\stmanifold{n_j}{1}  }(\mathbf y) -\iota_{\stmanifold{n_j}{1}}(\mathbf u^{k+1}_{j,i})    -\langle \tilde{\mathbf v}^{k+1}_{j,i}, \mathbf  y -\mathbf  u^{k+1}_{j,i} \rangle}{\|\mathbf  y-\mathbf   u^{k+1}_{j,i} \|}\geq 0.
\]
  As a result, \eqref{eq:global_conv:1} is true, which together with \eqref{eq:global_conv:0} shows that
\begin{equation*}\label{eq:global_conv:3}
\tilde{\mathbf v}^{k+1}_{j,i}  - \overline{\mathbf v}^{k+1}_{j,i} \in \partial_{\mathbf u_{j,i}} \tildeLtaualpha{k+1}{k} , ~1\leq j\leq d-t,~1\leq i\leq R.
  \end{equation*}
Let $\boldsymbol{0}$ denote the original. Then by using the triangle inequality and the boundeness of $\{\boldsymbol{ \sigma}^k,U^k,\mathcal T^k,\mathcal Y^k   \}$, and noticing the definition of $\deltaUjP{k+1}{k}$,  there must exist large enough constants $c_1,c_2>0$ only depending on $\tau,\alpha$, and the size of  $\{\boldsymbol{ \sigma}^k,U^k,\mathcal T^k,\mathcal Y^k   \}$, such that
  \begin{eqnarray}
 && {\rm dist}(\boldsymbol{0}, \partial_{\mathbf u_{j,i}} \tildeLtaualpha{k+1}{k} )  \nonumber\\
 \leq&&  \bignorm{   \tilde{\mathbf v}^{k+1}_{j,i} - \overline{\mathbf v}^{k+1}_{j,i}  } \nonumber\\
 \leq && c_1\bigxiaokuohao{ \sum^{d}_{j=1}\bigfnorm{ \deltaUj{k+1}{k}  } + \bigfnorm{ \deltaP{k+1}{k}  } + \bigfnorm{\deltaY{k+1}{k}}     }\nonumber\\
 \leq&& c_1\bigxiaokuohao{  \sum^{d}_{j=1}\bigfnorm{ \deltaUj{k+1}{k}  } + 2\bigfnorm{  \deltaP{k+1}{k}} + \bigfnorm{\deltaP{k}{k-1}} } \nonumber \\    
  \leq && c_2 \bigxiaokuohao{  \bigfnorm{ \deltaUjP{k+1}{k}  }  + \bigfnorm{ \deltaUjP{k}{k-1}}    },~1\leq j\leq d-t.  \label{eq:global_conv:6}
\end{eqnarray}

  On the other hand, for $d-t+1\leq j\leq d$, by noticing the definition of $\overline{ V}^{k+1}_j$, we have
  \[
\partial_{U_j} \tildeLtaualpha{k+1}{k}   = - \overline{V}^{k+1}_j + \partial \iota_{ \stmanifold{n_j}{R}  }(U^{k+1}_j).
  \]
  From the definition of $U^{k+1}_j$ in \eqref{prob:U_subproblem} and similar to the above argument, we can show that 
  $\tilde V^{k+1}_j \in \partial \iota_{\stmanifold{n_j}{R}}(U^{k+1}_j).  $ Thus
\begin{equation*}\label{eq:global_conv:4}
\tilde V^{k+1}_j - \overline V^{k+1}_j \in \partial_{ U_j} \tildeLtaualpha{k+1}{k}  ,~d-t+1\leq j\leq d.
  \end{equation*}
Similar to \eqref{eq:global_conv:6}, there exists a large enough constant $c_3>0$ such that
\begin{equation}
{\rm dist}(\boldsymbol{0}, \partial_{\mathbf u_{j,i}} \tildeLtaualpha{k+1}{k}  )   \leq c_3\bigxiaokuohao{  \bigfnorm{ \deltaUjP{k+1}{k}  }   + \bigfnorm{ \deltaUjP{k}{k-1}}   },~d-t+1\leq j\leq d.\label{eq:global_conv:7}
\end{equation}

We then consider 
\[\nabla_{\mathcal T} \tildeLtaualpha{k+1}{k} = \mathcal W^{k+1}\circledast\bigxiaokuohao{\mathcal T^{k+1} -\mathcal A} + \mathcal Y^{k+1} - \tau\bigxiaokuohao{ \llbracket \boldsymbol{ \sigma}^{k+1};U^{k+1}_j  \rrbracket - \mathcal T^{k+1} } + \frac{4}{\tau}\bigxiaokuohao{\mathcal T^{k+1} - \mathcal T^k}.\]
Note that  $\mathcal W^{k+1}$ and $\boldsymbol{ \sigma}^{k+1}$ above are only representations instead of variables, which represent   \eqref{prob:W_subproblem} and  \eqref{prob:sigma_subproblem}. From the expression of $\mathcal Y^{k+1}$ in \eqref{eq:conv:1}, we have
\begin{eqnarray}
\bigfnorm{\mathcal W^{k+1}\circledast\bigxiaokuohao{\mathcal T^{k+1} -\mathcal A } + \mathcal Y^{k+1}  } &=& \bigfnorm{ \bigxiaokuohao{\mathcal W^{k+1} - \mathcal W^k}\circledast\bigxiaokuohao{\mathcal T^{k+1} -\mathcal A}  }\nonumber\\
&\leq&\bigfnorm{\deltaP{k+1}{k}},\nonumber
\end{eqnarray}
where   the inequality follows from Proposition \ref{prop:E2}. On the other side,
\begin{eqnarray}
\tau\bigfnorm{  { \llbracket \boldsymbol{ \sigma}^{k+1};U^{k+1}_j  \rrbracket - \mathcal T^{k+1} } } &=& \tau\bigfnorm{ \llbracket\boldsymbol{\sigma}^{k+1};U^{k+1}_j  \rrbracket  - \llbracket\boldsymbol{\sigma}^{k};U^{k+1}_j  \rrbracket   + \llbracket\boldsymbol{\sigma}^{k };U^{k+1}_j  \rrbracket -\mathcal T^{k+1} } \nonumber\\
&\leq& \tau \bigfnorm{\llbracket\boldsymbol{\sigma}^{k+1};U^{k+1}_j  \rrbracket  - \llbracket\boldsymbol{\sigma}^{k};U^{k+1}_j  \rrbracket    } + \bigfnorm{\deltaY{k+1}{k}}\nonumber\\
&\leq& c_4\bigxiaokuohao{\bigfnorm{\deltaUjP{k+1}{k}  } + \bigfnorm{\deltaUjP{k}{k-1}}}, \label{eq:global_conv:10}
\end{eqnarray}
where $c_4>0$ is large enough. Combining the above pieces shows that there exists a large enough constant $c_5>0$ such that
\begin{equation}
\label{eq:global_conv:8}
\bigfnorm{\nabla_{\mathcal T} \tildeLtaualpha{k+1}{k}  } \leq c_5\bigxiaokuohao{  \bigfnorm{\deltaUjP{k+1}{k}  } + \bigfnorm{\deltaUjP{k}{k-1}} }.
\end{equation}
Next, it follows from \eqref{eq:global_conv:10} that
\begin{equation}
\label{eq:global_conv:9}
\bigfnorm{ \nabla_{\mathcal Y}\tildeLtaualpha{k+1}{k}  } = \bigfnorm{  { \llbracket \boldsymbol{ \sigma}^{k+1};U^{k+1}_j  \rrbracket - \mathcal T^{k+1} }   } \leq \frac{c_4}{\tau}\bigxiaokuohao{  \bigfnorm{\deltaUjP{k+1}{k}  } + \bigfnorm{\deltaUjP{k}{k-1}} }.
\end{equation}
Finally,
\begin{equation}
\label{eq:global_conv:11}
\bigfnorm{ \nabla_{\mathcal T^\prime} \tildeLtaualpha{k+1}{k}   } = \frac{4}{\tau}\bigfnorm{\deltaP{k+1}{k}}.
\end{equation}
Combining \eqref{eq:global_conv:6}, \eqref{eq:global_conv:7}, \eqref{eq:global_conv:8}, \eqref{eq:global_conv:9}, \eqref{eq:global_conv:11}, we get that there exists a large enough constant $c_0>0$ independent of $k$, such that
\begin{equation*}
 {\rm dist}(\boldsymbol{0}, \partial \tildeLtaualpha{k+1}{k} ) \leq c_0\bigxiaokuohao{ \bigfnorm{ \deltaUjP{k+1}{k}} + \bigfnorm{\deltaUjP{k}{k-1}}     },
\end{equation*}
as desired.
\end{proof}

Now we can present the proof concerning global convergence.
 \begin{proof}[Proof of Theorem \ref{th:global_conv}]
We have mentioned that $\{ \tildeLtaualpha{k+1}{k}  \}$  inherits the properties of $\{\tilde L_\tau^{k+1,k} \}$, i.e., it is bounded, nonincreasing and convergent. We denote its limit as   $\tilde L^*_{\tau,\alpha} = \lim_{k\rightarrow\infty} \tildeLtaualpha{k+1}{k}  = \tilde L_{\tau,\alpha}( U^*_j,\mathcal T^*,\mathcal Y^*,\mathcal T^*  )$ where $\{U^*_j,\mathcal T^*,\mathcal Y^* ,\mathcal T^* \}$ is a limit point. According to Definition \ref{def:kl} and Proposition \ref{prop:KL_L_tau_alpha}, there exist an $\epsilon_0>0$, a neighborhood of $\{U^*_j,\mathcal T^*,\mathcal Y^* ,\mathcal T^* \}$, and a 
continuous and concave function $\psi(\cdot):[0,\epsilon_0) \rightarrow\mathbb R_+$ such that for all   $\{U_j,\mathcal T,\mathcal Y,\mathcal T^\prime\}  \in\mathcal N$ satisfying  $\tilde L_{\tau,\alpha}^* < \tilde L_{\tau,\alpha}( U_j,\mathcal T,\mathcal Y,\mathcal T^\prime) <\tilde L_{\tau,\alpha}^* + \epsilon_0$, there holds
\begin{equation}\label{eq:kl_true}
\psi^\prime(\tilde L_{\tau,\alpha}( U_j,\mathcal T,\mathcal Y,\mathcal T^\prime) -\tilde L_{\tau,\alpha}^*  ){\rm dist}(0,\partial \tilde L_{\tau,\alpha}( U_j,\mathcal T,\mathcal Y,\mathcal T^\prime) \geq 1.
\end{equation}

Let $\epsilon_1 >0$ be such that
\begin{small}
\[ 
\mathbb B_{\epsilon_1} := \{ \bigxiaokuohao{U_j,\mathcal T,\mathcal Y,\mathcal T^\prime}\mid \|\bigfnorm{U_j-U^*_j} < \epsilon_1,1\leq j\leq d,\bigfnorm{\mathcal T-\mathcal T^*}< \epsilon_1, \bigfnorm{\mathcal Y-\mathcal Y^*}<2 \epsilon_1, \bigfnorm{\mathcal T^\prime - \mathcal T^*} <2\epsilon_1  \} \subset \mathcal N,
\]
\end{small}
and let $\mathbb B^{U_j,\mathcal T}_{\epsilon_1}:= \{ \bigxiaokuohao{U_j,\mathcal T}\mid \bigfnorm{U_j -U^*_j} < \epsilon_1,1\leq j\leq d,\bigfnorm{\mathcal T-\mathcal T^*}<\epsilon_1  \}$. From the stationary point system \eqref{eq:kkt}  and the expression of $\mathcal Y^{k+1}$ in \eqref{eq:conv:1}, we have
\begin{eqnarray}
\bigfnorm{\mathcal Y^{k} -\mathcal Y^* } &=& \bigfnorm{  \mathcal W^{k-1}\circledast\bigxiaokuohao{ \mathcal T^{k}  - \mathcal A} -  \mathcal W^*\circledast\bigxiaokuohao{\mathcal T^* - \mathcal A}   } \nonumber \\
&\leq& \bigfnorm{ \mathcal W^{k-1}\circledast\bigxiaokuohao{ \mathcal T^{k}  - \mathcal A}  -\mathcal W^k\circledast\bigxiaokuohao{ \mathcal T^k-\mathcal A}} +\bigfnorm{ \mathcal W^k\circledast\bigxiaokuohao{\mathcal T^k-\mathcal A}  -\mathcal W^*\circledast
\bigxiaokuohao{\mathcal T^*-\mathcal A} }   \nonumber\\
&=&   \bigfnorm{\deltaP{k}{k-1}} + \bigfnorm{\deltaP{k}{*}}  \label{eq:global:7}
\end{eqnarray}
where the last inequality follows from Propositions \ref{prop:E2} and \ref{prop:lipschitz_gradient}.  On the other hand, 
\begin{equation}\label{eq:global:9}
\bigfnorm{\mathcal T^{k-1} -\mathcal T^*} \leq \bigfnorm{\deltaP{k}{k-1}} + \bigfnorm{\deltaP{k}{*}}.
\end{equation}
As Theorem \ref{th:sub_convergence} shows that there exists $k_0>0$ such that for $k\geq k_0$, $\bigfnorm{\deltaP{k}{k-1}}<\epsilon_1$,  \eqref{eq:global:7} and \eqref{eq:global:9} tells us that if $k\geq k_0$ and $(U^k_j,\mathcal T^k)\in\mathbb B^{U_j,\mathcal T }_{\epsilon_1 }$, then $\{U^k_j,\mathcal T^k,\mathcal Y^k,\mathcal T^{k-1} \} \in \mathbb B_{\epsilon_1 } \subset\mathcal N$. Such $k_0$ must exist as $\{U^*_j,\mathcal T^*,\mathcal Y^* ,\mathcal T^* \}$ is a limit point.   In addition, denote $c_1:=\min\{\alpha/2,1/\tau  \}$; then there exists $k_1\geq k_0$ such that $  (U^{k_1}_j,  \mathcal T^{k_1} )  \in\mathbb B^{ U_j,\mathcal T  }_{\epsilon_1/2} $ and
\begin{equation}
\begin{split}
&   \frac{c_0}{2\sqrt{c_1}c_2}    \bigfnorm{\deltaUjP{k_1}{k_1-1}}  < \frac{\epsilon_1}{16},~    \frac{c_0}{2\sqrt{c_1}c_2}   \bigfnorm{\deltaUjP{k_1-1}{k_1-2}}   <\frac{\epsilon_1}{16},~   \frac{c_2}{2\sqrt{c_1}}   \psi( \tilde L_{\tau,\alpha}^{k_1,k_1-1} - L^*_{\tau,\alpha}  ) < \frac{\epsilon_1}{4},\\
&L^*_{\tau,\alpha} < \tilde L_{\tau,\alpha}^{k_1,k_1-1} < L^*_{\tau,\alpha} + \epsilon_0,\label{eq:parameters}
\end{split}
\end{equation}
where $c_0$ is the constant appeared in Lemma \ref{lem:bounded_partial_L}, and $c_2 $ is a constant such that $c_2 > 16c_0/\sqrt{c_1}$.

In what follows, we use induction method to show that $\bigxiaokuohao{U^k_j,\mathcal T^k}\in\mathbb B^{U_j,\mathcal T}_{\epsilon_1}$ for all $k > k_1$.
Since   $\psi(\cdot)$ in Definition \ref{def:kl}   is concave, it holds that for any $k$, 
\begin{equation}\label{eq:global:10}
\psi^\prime( \tilde L^{k,k-1}_{\tau,\alpha} - L^*_{\tau,\alpha}  )\left(  (\tilde L^{k,k-1}_{\tau,\alpha} - \tilde L^*_{\tau,\alpha}) - (\tilde L^{k+1,k}_{\tau,\alpha} - \tilde L^*_{\tau,\alpha}  ) \right) \leq \psi(\tilde L^{k,k-1}_{\tau,\alpha} -\tilde L^*_{\tau,\alpha}) - \psi( \tilde L^{k+1,k}_{\tau,\alpha} -\tilde L^*_{\tau,\alpha} );
\end{equation}
on the other side, from the   previous paragraph we see that  $(U^{k_1}_j,\mathcal T^{k_1})\in\mathbb B^{ U_j,\mathcal T  }_{\epsilon_1/2}$, $\{ U^{k_1}_j,\mathcal T^{k_1},\mathcal Y^{k_1},\mathcal T^{k_1-1}  \} \in \mathbb B_{\epsilon_1} \subset \mathcal N$, and so \eqref{eq:kl_true} holds at $\{U^{k_1}_j,\mathcal T^{k_1},\mathcal Y^{k_1},\mathcal T^{k_1-1} \}$. Recall $c_1=\min\{\alpha/2,1/\tau  \}$. From Lemma \ref{lem:suff_decrease_1} and the relation between $\tilde L_\tau$ and $\tilde L_{\tau,\alpha}$, we obtain  
\begin{eqnarray*}
	c_1\bigfnorm{\deltaUjP{k_1+1}{k}}^2  &\leq& \tilde L^{k_1,k_1-1}_{\tau,\alpha} - \tilde L_{\tau,\alpha}^{k_1+1,k_1} \\
	&\leq& \frac{\psi(\tilde L^{k_1,k_1-1}_{\tau,\alpha} - \tilde L^*_{\tau,\alpha}) - \psi( \tilde L^{k_1+1,k_1}_{\tau,\alpha} -\tilde L^*_{\tau,\alpha} )}{\psi^\prime( \tilde L^{k_1,k_1-1}_{\tau,\alpha} - \tilde L^*_{\tau,\alpha}  )} \\
	&\leq& c_2\left( \psi(\tilde L^{k_1,k_1-1}_{\tau,\alpha} -\tilde L^*_{\tau,\alpha}) - \psi( \tilde L^{k_1+1,k_1}_{\tau,\alpha} -\tilde L^*_{\tau,\alpha} ) \right) \cdot c_2^{-1}{\rm dist}(0, \partial  \tilde L^{k_1,k_1-1}_{\tau,\alpha} ),
\end{eqnarray*}
where   the second inequality is due to \eqref{eq:global:10} while the last one comes from \eqref{eq:kl_true}. Using $\sqrt{ab}\leq \frac{a+b}{2}$ for $a\geq 0,b\geq 0$, invoking \eqref{eq:global_conv:12_dist_2_partial} and noticing the range in \eqref{eq:parameters}, we obtain
\begin{eqnarray*}
\sqrt{c_1} \bigfnorm{\deltaUjP{k_1+1}{k}} &\leq& \frac{c_2}{2}\left( \psi(\tilde L^{k_1,k_1-1}_{\tau,\alpha} -\tilde L^*_{\tau,\alpha}) - \psi( \tilde L^{k_1+1,k_1}_{\tau,\alpha} -\tilde L^*_{\tau,\alpha} ) \right)\nonumber\\
&&~~~~~~~~~~~~~~~~~~~~ + \frac{c_0 }{2c_2}\bigxiaokuohao{ \bigfnorm{ \deltaUjP{k_1}{k_1-1}  } + \bigfnorm{\deltaUjP{k_1-1}{k_1-2}}   }\label{eq:global:12}\\
&<&\frac{ \sqrt{c_1}\epsilon_1}{4} + \frac{ \sqrt{c_1}\epsilon_1}{8} < \frac{\sqrt{c_1}\epsilon_1}{2},\label{eq:global:11}\nonumber
\end{eqnarray*}
and so
\begin{equation*}
\bigfnorm{\deltaUjP{k_1+1}{*}}\leq \bigfnorm{\deltaUjP{k_1+1}{k_1}} + \bigfnorm{\deltaUjP{k_1}{*}} < \frac{\epsilon_1}{2} + \frac{\epsilon_1}{2}=\epsilon_1,
\end{equation*}
namely, $ (U^{k_1+1}_j,\mathcal T^{k_1+1}) \in\mathbb B^{U_j,\mathcal T }_{\epsilon_1}$.  

Now assume that $(U^k_j,\mathcal T^k)\in\mathbb B^{U_j,\mathcal T  }_{\epsilon_1}$ for $k=k_1,\ldots,K$. This implies that \eqref{eq:kl_true} is true at $\{U^k_j,\mathcal T^k,\mathcal Y^k,\mathcal T^{k-1}  \}$, and similarly to the above analysis, we have
\begin{equation}
\sqrt{c_1} \bigfnorm{\deltaUjP{k+1}{k}}  \leq  \frac{c_2}{2}\left( \psi(\tilde L^{k,k-1}_{\tau,\alpha} - \tilde L^*_{\tau,\alpha}) - \psi( \tilde L^{k+1,k}_{\tau,\alpha} -\tilde L^*_{\tau,\alpha} ) \right) + \frac{c_0}{2c_2}\left(\bigfnorm{\deltaUjP{k}{k-1}}+ \bigfnorm{\deltaUjP{k-1}{k-2}}  \right),~k=k_1,\ldots,K. \label{eq:global:13}
\end{equation}
We then show that $(U^{K+1}_j,\mathcal T^{K+1})\in\mathbb B^{ U_j,\mathcal T}_{\epsilon_1}$. Summing \eqref{eq:global:13} for $k=k_1,\ldots,K$ yields
\begin{eqnarray}
\sqrt{c_1} \sum^K_{k=k_1}\bigfnorm{\deltaUjP{k+1}{k}} &\leq& \frac{c_2}{2}\left( \psi(\tilde L^{k_1,k_1-1}_{\tau,\alpha} - \tilde L^*_{\tau,\alpha}) - \psi( \tilde L^{K+1,K}_{\tau,\alpha} -\tilde L^*_{\tau,\alpha} ) \right)   + \frac{c_0}{2c_2}\sum^K_{k=k_1} \left( \bigfnorm{ \deltaUjP{k}{k-1} } +  \bigfnorm{\deltaUjP{k-1}{k-2}}  \right)\nonumber\\
&\leq& \frac{c_2}{2}\left( \psi(\tilde L^{k_1,k_1-1}_{\tau,\alpha} - \tilde L^*_{\tau,\alpha}) - \psi( \tilde L^{K+1,K}_{\tau,\alpha} -\tilde L^*_{\tau,\alpha} ) \right)  \nonumber\\
&&+ \frac{c_0}{c_2}\sum^{K-1}_{k=k_1}\bigfnorm{\deltaUjP{k+1}{k}} +   \frac{2c_0}{c_2}\bigfnorm{\deltaUjP{k_1}{k_1-1}}+ \frac{c_0}{c_2}\bigfnorm{\deltaUjP{k_1-1}{k_1-2}}   .\label{eq:global:14}
\end{eqnarray}
Rearranging the terms, noticing \eqref{eq:parameters} and noticing that $\frac{c_2}{c_0}> \frac{\sqrt{c_1}}{16}$, we have
\begin{equation*}
\frac{15\sqrt{c_1}}{16} \sum^K_{k=k_1} \bigfnorm{\deltaUjP{k+1}{k}}  \leq  \frac{\sqrt{c_1}}{4}\epsilon_1 +   \frac{\sqrt c_1\epsilon_1}{16} +   \frac{\sqrt c_1\epsilon_1}{16}, 
\end{equation*}
and so 
\begin{eqnarray*}
\bigfnorm{\deltaUjP{K+1}{*}}&\leq& \bigfnorm{\deltaUjP{K+1}{k_1}}+ \bigfnorm{\deltaUj{k_1}{*}}\nonumber\\
&<& \sum^K_{k=k_1}\bigfnorm{\deltaUjP{k+1}{k}}+ \frac{\epsilon_1}{2}\nonumber\\
&< &   \frac{3\epsilon_1}{8} + \frac{\epsilon_1}{2} < \epsilon_1.
\end{eqnarray*}
Thus induction method implies that $(U^k_j,\mathcal T^k)\in\mathbb B^{U_j,\mathcal T}_{\epsilon_1}$ for all $k\geq  k_1$, i.e.,  $\{U^k_j,\mathcal T^k,\mathcal Y^k,\mathcal T^{k-1}   \}\in \mathcal N$, $k\geq k_1$. As a result, \eqref{eq:global:13} holds for all $k\geq k_1$, so does \eqref{eq:global:14}. Therefore, letting $K\rightarrow\infty$ in \eqref{eq:global:14} yields 
\[
\sum^{\infty}_{k=1}\bigfnorm{\deltaUjP{k+1}{k}} <+\infty,
\]
which shows that $\{ U_j^k,\mathcal T^k\}$ is a Cauchy sequence and hence converges. Since $(U^*_j,\mathcal T^*)$ in Theorem \ref{th:sub_convergence} is a limit point, the whole sequence converges to $(U^*_j,\mathcal T^*)$. This completes the proof.
\end{proof}

\section{Numerical Experiments}\label{sec:numer}
We evaluate the robustness of model \eqref{prob:robust_orth_main} solved by HQ-ADMM in this section using synthetic and real data. The least squares based model \eqref{prob:obj_ls} is used as a comparison. \eqref{prob:obj_ls} is solved by the alternating least squares (ALS) method.
All the   computations are conducted on an Intel i7-7770 CPU desktop computer with 32 GB of RAM. The supporting software is Matlab R2015b.  The Matlab  package Tensorlab  \cite{tensorlab2013} is employed for tensor operations.  The Matlab code of HQ-ADMM is available   at \url{https://github.com/yuningyang19/hqadmm_rota}.

The stopping criterion for HQ-ADMM is 
$ \bigjueduizhi{\bigfnorm{\bigllbracket{\boldsymbol{ \sigma}^{k+1};U^{k+1}_j} -\mathcal A} - \bigfnorm{ \bigllbracket{\boldsymbol{ \sigma}^k;U^k_j} - \mathcal A  } }\leq 10^{-6}$ 
or $k\geq 2000$ for practical reasons. The parameter $\alpha$ in HQ-ADMM is set to $10^{-8}$,   $\tau\in \{0.7,1 \}$;      $\delta=0.05$.

\paragraph{Synthetic data}
We consider randomly generated tensors contaminated by different kinds of noises listed in the following
\begin{itemize}
	\item $\mathcal A = \mathcal A_0/\bignorm{\mathcal A_0}_F + \beta\cdot\mathcal N/\|\mathcal N\|_F$, where $\mathcal A_0$ is the ground truth tensor specified later, and $\mathcal N$ denotes the Cauchy noise, with scale parameter $\delta=0.05$. $\beta=0.5$;
	\item  $\mathcal A = \mathcal A_0/\bignorm{\mathcal A_0}_F + \mathcal O$. Here $\mathcal O$ denotes sparse outliers, with sparsity $0.1$, i.e., $10\%$ of the entries of $\mathcal A_0$ are contaminated by outliers. Outliers are drawn uniformly from $[0,10]$;
	\item $\mathcal A = \mathcal A_0/\bignorm{\mathcal A_0}_F + \beta\cdot\mathcal N/\|\mathcal N\|_F$, where $\mathcal N$ denotes Gaussian noise, with  $\beta=0.1$.
\end{itemize}
The ground truth tensor    $\mathcal A_0 = \sum^R_{i=1} \sigma_i\bigotimesu $,  where   $U_j$   are randomly drawn from a uniformly distribution in $[-1,1]$.  $U_j$, $d-t+1\leq j\leq d$, are then made to be columnwisely orthonormal, while the remaing $U_j$ are columnwisely normalized. $\sigma_i$ are drawn from Gaussian distribution. 
For convenience,  we set $d=3$ or $4$, $n_1=\cdots =n_d$, and $R=5$ in all the experiments in this part. The initializers for HQ-ADMM and ALS are randomly generated. The reported results are averaged over 50 instances for each case.

\begin{table*}[h] \footnotesize
	\renewcommand\arraystretch{0.9}
	
	\begin{floatrow}
		\capbtabbox{
			
			\setlength{\tabcolsep}{0.4mm}			
			
 \begin{mytabular1}{rr|rrr|rrr}
	\toprule
	&              &         \multicolumn{3}{c}{HQ-ADMM for  \eqref{prob:robust_orth_main}}       & \multicolumn{3}{c}{ALS for \eqref{prob:obj_ls}}    \\
	\toprule
	\multicolumn{1}{c}{$n$} & \multicolumn{1}{c}{$(d,t)$}   & \multicolumn{1}{c}{err.} & \multicolumn{1}{c}{iter.} & \multicolumn{1}{c}{time} & \multicolumn{1}{c}{err.} & \multicolumn{1}{c}{iter.} & \multicolumn{1}{c}{time} \\
	\toprule
	10    & $(3, 1)$     & 5.57E-02 & 395   & 0.16  & 4.29E-01 & 149   & 0.04  \\
	20    & $(3, 1)$     & 4.66E-02 & 315   & 0.21  & 4.20E-01 & 147   & 0.05  \\
	50    & $(3, 1)$      & 4.30E-02 & 45    & 0.09  & 4.33E-01 & 309   & 0.27  \\
	80    & $(3, 1)$     & 3.05E-02 & 71    & 0.77  & 4.31E-01 & 190   & 1.16  \\
	90    & $(3, 1)$     & 3.04E-02 & 47    & 0.76  & 4.29E-01 & 152   & 1.28  \\
	100   & $(3, 1)$      & 3.21E-02 & 86    & 1.62  & 4.41E-01 & 210   & 1.82  \\
	10    & $(3, 2)$      & 5.25E-02 & 453   & 0.19  & 3.84E-01 & 33    & 0.01  \\
	20    & $(3, 2)$    & 2.93E-02 & 137   & 0.10  & 4.12E-01 & 17    & 0.01  \\
	60    & $(3, 2)$    & 2.25E-02 & 200   & 1.03  & 4.42E-01 & 11    & 0.04  \\
	80    & $(3, 2)$    & 2.20E-02 & 58    & 0.60  & 4.18E-01 & 11    & 0.07  \\
	90    & $(3, 2)$     & 2.02E-02 & 136   & 2.11  & 4.33E-01 & 14    & 0.11  \\
	100   & $(3, 2)$     & 2.57E-02 & 96    & 1.84  & 4.23E-01 & 10    & 0.09  \\
	80    & $(3, 3)$     & 1.39E-02 & 35    & 0.34  & 1.41E+00 & 2     & 0.02  \\
	100   & $(3, 3)$     & 2.08E-02 & 89    & 1.69  & 1.41E+00 & 2     & 0.03  \\
	10    & $(4, 1)$     & 3.86E-02 & 64    & 0.08  & 4.12E-01 & 341   & 0.21  \\
	20    & $(4, 1)$    & 7.98E-02 & 40    & 0.16  & 4.45E-01 & 613   & 1.02  \\
	30    & $(4, 1)$     & 7.37E-02 & 28    & 0.71  & 4.25E-01 & 485   & 6.55  \\
	40    & $(4, 1)$    & 5.08E-02 & 25    & 1.62  & 4.47E-01 & 637   & 16.68  \\
	10    & $(4, 2)$     & 4.98E-02 & 75    & 0.09  & 4.56E-01 & 299   & 0.19  \\
	20    & $(4, 2)$    & 1.11E-01 & 53    & 0.20  & 4.73E-01 & 527   & 0.94  \\
	30    & $(4, 2)$     & 7.33E-02 & 36    & 1.09  & 4.76E-01 & 394   & 6.06  \\
	40    & $(4, 2)$     & 6.85E-02 & 27    & 1.75  & 4.70E-01 & 705   & 19.25  \\
	10    & $(4, 3)$    & 9.57E-02 & 100   & 0.12  & 4.83E-01 & 664   & 0.41  \\
	20    & $(4, 3)$    & 8.60E-02 & 69    & 0.27  & 5.00E-01 & 707   & 1.04  \\
	30    & $(4, 3)$    & 1.29E-01 & 35    & 0.98  & 5.18E-01 & 645   & 9.72  \\
	40    & $(4, 3)$    & 1.40E-01 & 30    & 1.86  & 5.41E-01 & 878   & 22.68  \\
	\bottomrule
\end{mytabular1}%

		}{
			\caption{\footnotesize Comparison of HQ-ADMM for \eqref{prob:robust_orth_main} and ALS for \eqref{prob:obj_ls} when the ground truth tensor is contaminated by Cauchy noise.}
			\label{tab:cauchy}
		}
		\capbtabbox{
			
			\setlength{\tabcolsep}{0.4mm}

	\begin{mytabular1}{rr|rrr|rrr}
	\toprule
	&              &         \multicolumn{3}{c}{HQ-ADMM for  \eqref{prob:robust_orth_main}}       & \multicolumn{3}{c}{ALS for \eqref{prob:obj_ls}}    \\
	\toprule
	\multicolumn{1}{r}{$n$} & \multicolumn{1}{l}{$(d,t)$}   & \multicolumn{1}{r}{err.} & \multicolumn{1}{r}{iter.} & \multicolumn{1}{r}{time} & \multicolumn{1}{r}{err.} & \multicolumn{1}{r}{iter.} & \multicolumn{1}{r}{time} \\
	\toprule
	10    & $(3, 1)$    & 4.54E-01 & 89    & 0.04  & 1.40E+00 & 150   & 0.04  \\
	20    & $(3, 1)$     & 5.95E-02 & 46    & 0.04  & 1.41E+00 & 251   & 0.09  \\
	50    & $(3, 1)$    & 1.99E-02 & 31    & 0.10  & 1.41E+00 & 757   & 0.95  \\
	80    & $(3, 1)$    & 2.21E-02 & 27    & 0.55  & 1.41E+00 & 1456  & 12.17  \\
	90    & $(3, 1)$     & 3.52E-02 & 28    & 0.70  & 1.41E+00 & 1204  & 11.59  \\
	100   & $(3, 1)$     & 2.82E-02 & 31    & 0.91  & 1.41E+00 & 1390  & 15.44  \\
	10    & $(3, 2)$     & 4.32E-01 & 56    & 0.03  & 1.41E+00 & 120   & 0.04  \\
	20    & $(3, 2)$     & 6.13E-02 & 35    & 0.04  & 1.41E+00 & 314   & 0.15  \\
	50    & $(3, 2)$     & 7.50E-03 & 25    & 0.07  & 1.41E+00 & 592   & 0.69  \\
	80    & $(3, 2)$    & 7.40E-03 & 25    & 0.42  & 1.41E+00 & 820   & 6.05  \\
	90    & $(3, 2)$     & 6.66E-03 & 26    & 0.65  & 1.41E+00 & 828   & 7.80  \\
	100   & $(3, 2)$     & 8.16E-03 & 27    & 0.90  & 1.41E+00 & 928   & 11.99  \\
	80    & $(3, 3) $   & 6.08E-03 & 25    & 0.42  & 1.41E+00 & 2     & 0.02  \\
	100   & $(3, 3) $    & 6.72E-03 & 27    & 0.80  & 1.41E+00 & 2     & 0.04  \\
	10 	& $(4, 1)$ &1.04E-01 &	76 & 0.23 & 1.42E+00 & 187 & 0.14\\ 
	20 	& $(4, 1)$ & 2.91E-02 & 34 & 0.28 & 1.41E+00 & 439 & 1.02\\ 
	30 	& $(4, 1)$ & 4.40E-02 & 28 & 1.06 & 1.41E+00 & 1173 & 18.40\\ 
	40 	& $(4, 1)$ & 6.09E-02 & 27 & 2.00 & 1.41E+00 & 885 & 26.09\\ 
	10    & $(4, 2)$     & 1.31E-01 & 67    & 0.08  & 1.41E+00 & 246   & 0.16  \\
	20    & $(4, 2)$    & 5.23E-02 & 28    & 0.13  & 1.41E+00 & 729   & 1.12  \\
	30    & $(4, 2)$     & 6.17E-02 & 27    & 0.85  & 1.41E+00 & 697   & 12.68  \\
	40    & $(4, 2)$      & 3.36E-02 & 29    & 1.88  & 1.41E+00 & 1047  & 29.12  \\
	10    & $(4, 3)$    & 1.40E-01 & 64    & 0.08  & 1.41E+00 & 208   & 0.13  \\
	20    & $(4, 3)$    & 8.14E-02 & 29    & 0.12  & 1.41E+00 & 622   & 0.92  \\
	30    & $(4, 3)$     & 8.45E-02 & 38    & 1.15  & 1.41E+00 & 900   & 14.85  \\
	40    & $(4, 3)$     & 1.13E-01 & 30    & 2.12  & 1.41E+00 & 846   & 24.38  \\
	\bottomrule
\end{mytabular1}%

		}{
			\caption{\footnotesize Comparison of HQ-ADMM for \eqref{prob:robust_orth_main} and ALS for \eqref{prob:obj_ls} when the ground truth tensor is contaminated by outliers.}
			\label{tab:outliers}
		}
	\end{floatrow}
\end{table*}

Comparisons of HQ-ADMM for solving \eqref{prob:robust_orth_main} and ALS for solving \eqref{prob:obj_ls} with Cauchy noise are reported in Table  \ref{tab:cauchy},  where ${\rm err.} = \bigfnorm{ \mathcal A_0/\bigfnorm{\mathcal A_0} - \mathcal A^*/\bigfnorm{\mathcal A^*}  }$, with $\mathcal A^*=\bigllbracket{\boldsymbol{ \sigma}^*;U^*_j }$ the tensor generated by the algorithm. ``iter.' denotes the number of iterates, and ``time'' stands for the CPU time consumed by the algorithm. From the ``err.'' columns, we see that in all cases,   HQ-ADMM performs much better than ALS; in particular, ``err.'' of HQ-ADMM is smaller than $0.1$ in almost all cases, which confirms that the proposed model and algorithm are consistent with Cauchy noise. Considering the efficiency, we see that HQ-ADMM all converges within $500$ iterates, and it consumes $1\sim 2$ seconds. Comparing with ALS,
when $d=3$, ALS is more efficient in most cases, while HQ-ADMM outperforms ALS when $d=4$. Thus HQ-ADMM is efficient.

The cases contaminated by outliers are reported in Table \ref{tab:outliers}, from which we can still observe that HQ-ADMM for solving \eqref{prob:robust_orth_main} is consistent with outliers, owing  to the redescending property of the Cauchy loss.  HQ-ADMM outperforms ALS in terms of the iterates and CPU time.

The cases with Gaussian noise are reported in Table \ref{tab:gaussian}. It is known that model \eqref{prob:obj_ls} is consistent with Gaussian noise, which can be seen from the table. We also observe that \eqref{prob:robust_orth_main} is consistent with Gaussian noise from the third column, although the results are slightly worse than \eqref{prob:obj_ls}, as reported in the table. However, it is interesting to see that in some cases, namely, $(n,d,t) = (80,3,1), (30,4,1),(40,4,1),(20,4,2),(30,4,3)$, HQ-ADMM for   \eqref{prob:robust_orth_main} is slightly better than ALS for \eqref{prob:obj_ls}. HQ-ADMM still shows its efficiency, and is more stable than ALS, as ALS needs much more iterates when $t=1$. 

\begin{table}[htbp]
	\centering
	\caption{\footnotesize Comparison of HQ-ADMM for \eqref{prob:robust_orth_main} and ALS for \eqref{prob:obj_ls} when the ground truth tensor is contaminated by Gaussian noise.}
	\begin{mytabular}{rrrrrrrr}
		\toprule
		&              &         \multicolumn{3}{c}{HQ-ADMM for  \eqref{prob:robust_orth_main}}       & \multicolumn{3}{c}{ALS for \eqref{prob:obj_ls}}    \\
		\toprule
		\multicolumn{1}{c}{$n$} & \multicolumn{1}{c}{$(d,t)$}   & \multicolumn{1}{c}{err.} & \multicolumn{1}{c}{iter.} & \multicolumn{1}{c}{time} & \multicolumn{1}{c}{err.} & \multicolumn{1}{c}{iter.} & \multicolumn{1}{c}{time} \\
		\toprule
		10    & $(3, 1)$     & 4.51E-02 & 198   & 0.09  & 4.09E-02 & 676   & 0.18  \\
		20    & $(3, 1)$     & 3.62E-02 & 53    & 0.04  & 2.73E-02 & 564   & 0.19  \\
		50    & $(3, 1)$     & 2.24E-02 & 30    & 0.08  & 2.18E-02 & 550   & 0.58  \\
		80    & $(3, 1)$     & 2.14E-02 & 34    & 0.57  & 2.72E-02 & 716   & 5.78  \\
		90    & $(3, 1)$     & 2.70E-02 & 33    & 0.79  & 2.44E-02 & 696   & 6.69  \\
		100   & $(3, 1)$     & 2.79E-02 & 34    & 0.98  & 2.28E-02 & 712   & 7.75  \\
		10    & $(3, 2)$     & 3.89E-02 & 296   & 0.13  & 3.48E-02 & 16    & 0.01  \\
		20    & $(3, 2)$    & 2.15E-02 & 65    & 0.05  & 1.87E-02 & 17    & 0.01  \\
		50    & $(3, 2)$     & 7.99E-03 & 24    & 0.07  & 7.67E-03 & 14    & 0.02  \\
		80    & $(3, 2)$    & 4.90E-03 & 24    & 0.40  & 4.82E-03 & 20    & 0.15  \\
		90    & $(3, 2)$     & 4.68E-03 & 25    & 0.60  & 4.34E-03 & 41    & 0.40  \\
		100   & $(3, 2)$     & 3.85E-03 & 24    & 0.72  & 3.85E-03 & 7     & 0.10  \\
		10    & $(4, 1)$     & 1.01E-01 & 673   & 0.83  & 8.62E-02 & 613   & 0.42  \\
		20    & $(4, 1)$      & 7.46E-02 & 67    & 0.31  & 6.21E-02 & 699   & 1.33  \\
		30    & $(4, 1)$      & 6.22E-02 & 29    & 1.05  & 6.61E-02 & 692   & 11.90  \\
		40    & $(4, 1)$      & 8.68E-02 & 27    & 1.92  & 1.11E-01 & 858   & 24.49  \\
		10    & $(4, 2)$     &  1.39E-02 &	45 	& 0.15 	& 1.74E-02	& 20 &	0.02   \\
		20    & $(4, 2)$     &  4.75E-03 &	23 	& 0.20 	& 9.09E-03	& 17 	& 0.05   \\
		30    & $(4, 2)$     & 5.42E-03	& 26 &	0.91 	& 2.71E-03	& 14 	& 0.25  \\
		40    & $(4, 2)$      &   2.26E-03 &	26 	& 2.10 	& 1.96E-03	& 41 &	1.24    \\
		10    & $(4, 3)$      &  1.29E-02 &	48 	& 0.17 	& 1.23E-02	& 10 	& 0.01   \\
		20    & $(4, 3)$     &  4.93E-03 &	24 &	0.21 &	4.73E-03 &	10 	& 0.04   \\
		30    & $(4, 3)$     &  2.72E-03	& 25 	& 0.98 &	2.88E-03 &	30 &	0.53   \\
		40    & $(4, 3)$      &  1.95E-03 &	26 &	2.15 &	1.92E-03 &	21 	& 0.67   \\
		\bottomrule
	\end{mytabular}%
	\label{tab:gaussian}%
\end{table}%

\begin{table}[htbp]
	\centering
	\caption{\footnotesize   HQ-ADMM for video surveillance with different $R$. The last column shows the compressed ratio of the compressed background factors $D,U,V$  to the sum of background frames $B_r$, $1\leq r\leq l$.}
	\begin{mytabular1}{ccccc}
		\toprule
		\multicolumn{1}{c}{$R$} & iter. &  time &  $\frac{R(1000+144+176)}{1000*144*176}$\\
		\toprule
10 	&43 &	33.86 
& 0.05\%\\
20 	&31 &	26.02 
&0.1\%\\
30 	&26 &	21.58 
&0.16\%\\
40 	&43 &	38.13 
&0.21\%\\
50 	&31 &	28.78 &0.26\%\\
		\bottomrule
\end{mytabular1}%
\label{tab:fb_extraction}%
\end{table}%

 \begin{figure} 
	\centering

	\renewcommand*{\arraystretch}{0.5}
	\setlength{\tabcolsep}{1pt}

	\begin{tabular}{ccccccc}

		\includegraphics[width=0.8in]
		{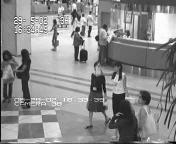}
		
		&

		\includegraphics[width=0.8in]
			{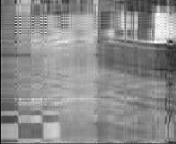}
		
		&
		\includegraphics[width=0.8in]
			{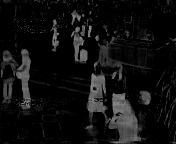}
			
		&
		\includegraphics[width=0.8in]
			{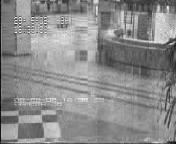}
			
					&
			\includegraphics[width=0.8in]
			{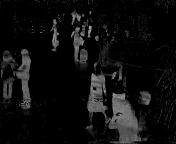}
			
					&
			\includegraphics[width=0.8in]
			{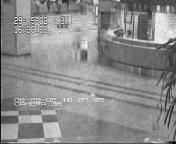}
			
					&
			\includegraphics[width=0.8in]
			{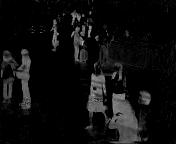}
		\\

		\includegraphics[width=0.8in]
{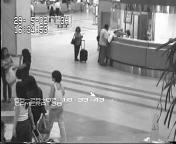}

&

\includegraphics[width=0.8in]
{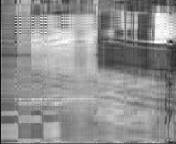}

&
\includegraphics[width=0.8in]
{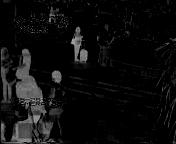}

&
\includegraphics[width=0.8in]
{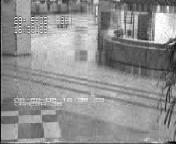}

&
\includegraphics[width=0.8in]
{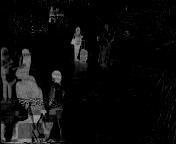}

&
\includegraphics[width=0.8in]
{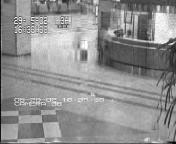}

&
\includegraphics[width=0.8in]
{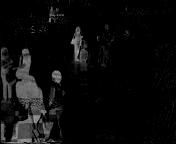}
\\

		\includegraphics[width=0.8in]
{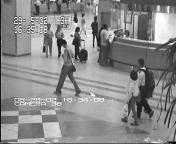}

&

\includegraphics[width=0.8in]
{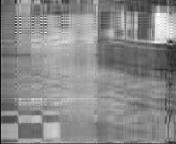}

&
\includegraphics[width=0.8in]
{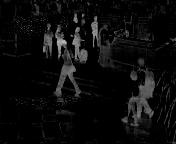}

&
\includegraphics[width=0.8in]
{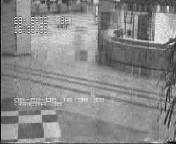}

&
\includegraphics[width=0.8in]
{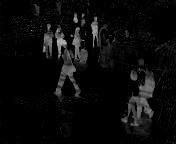}

&
\includegraphics[width=0.8in]
{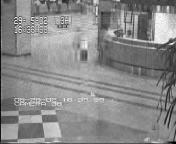}

&
\includegraphics[width=0.8in]
{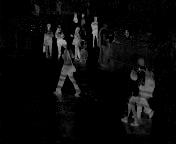}
\\

		\includegraphics[width=0.8in]
{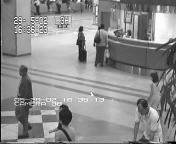}

&

\includegraphics[width=0.8in]
{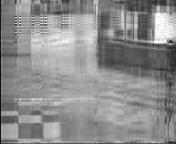}

&
\includegraphics[width=0.8in]
{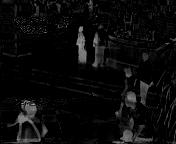}

&
\includegraphics[width=0.8in]
{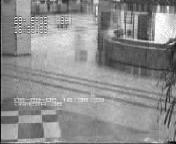}

&
\includegraphics[width=0.8in]
{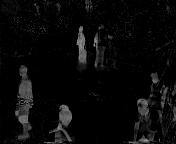}

&
\includegraphics[width=0.8in]
{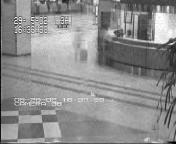}

&
\includegraphics[width=0.8in]
{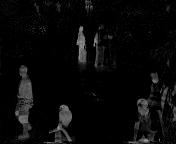}
\\

		\includegraphics[width=0.8in]
{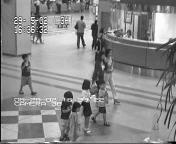}

&

\includegraphics[width=0.8in]
{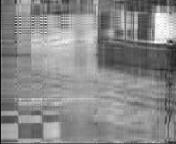}

&
\includegraphics[width=0.8in]
{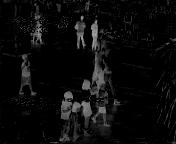}

&
\includegraphics[width=0.8in]
{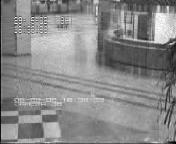}

&
\includegraphics[width=0.8in]
{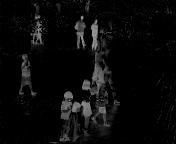}

&
\includegraphics[width=0.8in]
{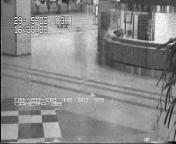}

&
\includegraphics[width=0.8in]
{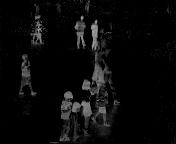}
\\
		
		\footnotesize (a) &\footnotesize (b) &\footnotesize (c) &\footnotesize (d)  &\footnotesize (e)  &\footnotesize (f)  &\footnotesize (g)     
	\end{tabular}

	\caption{Some extracting frames by HQ-ADMM from the video airport. Column (a): The original frames; Columns (b) and (c): Extracted with $R=10$;  Columns (d) and (e):  Extracted with $R=30$; Columns (f) and (g): Extracted with $R=50$. }\label{fig:video}
\end{figure}

\paragraph{Simultaneous  foreground-background extraction and compression}
Foreground-background extraction finds applications in video surveillance, where the aim is to detect moving objects  such as human beings from static background. As the background changes little in the video, it is reasonable to project the background frames to a low dimensional subspace to compress the data. We show how this problem can be fitted
into our model \eqref{prob:robust_orth_main}. Assume that a gray video consists of $l$ frames, each of size $m\times n$, resulting into a third-order tensor $\mathcal A\in\mathbb R^{l\times m\times n}$. Let $A_i$ denotes its $i$-th frame. Our goal is to decompose it as $A_r = B_r + F_r$, in which $B_r$ and $F_r$ denote  the back-/foreground frames, respectively. Under the assumption that $B_r$'s lie in a low dimensional subspace with commonalities,  we write $B_r = UD_r V^\top = \sum^R_{i=1} (D_r)_{ii} \mathbf u_i\mathbf v_i^\top$, $1\leq r\leq l$, where $U=[\mathbf u_1,\ldots,\mathbf u_R],V=[\mathbf v_1,\ldots,\mathbf v_R]$ are orthonormal matrices,   $D_r$ is   diagonal, and $R$ is a parameter. On the other hand, the foreground is often sparse and can be recognized as outliers. Therefore, the Cauchy loss can be employed to control the effect of outliers. Denoting 
\[
\phi_\delta(A_r- UD_r V^\top):=   \sum^{m,n}_{s=1,t=1}  \nolimits \frac{\delta^2}{2}\log\bigxiaokuohao{ 1+ \bigxiaokuohao{ (A_r)_{st} - (UD_r V^\top  )_{st}     }^2/\delta^2   },
\]
the problem can be modeled as
\[
\min_{U^\top U=I,V^\top V=I}\nolimits \sum^l_{r=1}\nolimits \phi_\delta(A_r- UD_r V^\top) .
\]
If we further  denote $D\in\mathbb R^{l\times R}$ where the $r$-th row is exactly the diagonal entries of $D_r$, the it can be written in the form of \eqref{prob:robust_orth_main}, i.e., 
\[
\min_{U^\top U=I,V^\top V=I}\nolimits  \boldsymbol{ \Phi}_\delta\bigxiaokuohao{ \mathcal A - \bigllbracket{ D,U,V   }   },
\]
where $\boldsymbol{ \sigma}$   is absorbed into $D$.

The tested video ``airport'' was downloaded from \url{http://perception.i2r.a-star.edu.sg/bk_model/bk_index.html}. The video consists of $4583$ frames, each of size $144\times 176$. We use $1000$ frames, resulting into a tensor $\mathcal A\in\mathbb R^{1000\times 144\times 176}$. $\mathcal A$ is then normalized for conveniently choosing parameters, where we set $\delta=0.05$, $\tau=1$, and $\alpha = 10^{-8}$. The parameter $R$ varies in $\{10,20,30,40,50\}$. The quantitative results are reported in Table \ref{tab:fb_extraction}, in which we can see that HQ-ADMM stops around $30\sim 40$ iterates, and consumes around $30$ seconds, which demonstrates the efficiency of the algorithm. The last column shows the compressed ratio of the compressed background factors $D,U,V$ to the sum of background frames $B_r$, $1\leq r\leq l$, from which we observe that the ratio is very high, resulting into low storage space. Some extracted frames with $R\in\{10,30,50\}$ are illustrated in Fig. \ref{fig:video}. From the figures, we see that even when $R=10$, HQ-ADMM can successfully seperate the back-/foreground; of course, when $R\geq 30$, the extrated frames are of higher quality, in that the background frames reconstructed from $UD_r V^\top$ are more clear. 
 
\section{Conclusions}\label{sec:conclusion}
 Heavy-tailed noise and outliers often contaminate real-world data. In the context of   tensor canonical polyadic approximation problem with one or more latent factor matrices having orthonormal columns, most existing models rely on the least squares loss, which is not resistant to heavy-tailed noise or outliers. To gain robustness, a Cauchy loss based robust orthogonal tensor approximation model was proposed in this work. To efficiently solve this model, by exploring its half-quadratic property, a new algorithm, termed as HQ-ADMM, was developed under the framework of alternating direction method of multipliers. Its global convergence was then established, thanks to some nice properties of the Cauchy loss. Numerical experiments on synthetic as well as real data demonstrate the efficiency and robustness of the proposed model and algorithm. In future work, it would be interesting to incorporate other robust losses in the orthogonal tensor approximation problem and to apply HQ-ADMM to solve other Cauchy loss based problems, as noted in Remark \ref{rmk:hq_admm}.

   \bibliography{nonconvex,tensor,robust,orth_tensor,TensorCompletion1}
\bibliographystyle{plain}

\end{document}